\documentclass{article}
\usepackage[small,nohug,heads=vee]{diagrams}
\diagramstyle[labelstyle=\scriptstyle]
\usepackage{extarrows}
\usepackage{amscd}
\usepackage{amssymb}
\usepackage{color}

\usepackage{titlesec}
\usepackage{bm}
\usepackage{times}
\usepackage{fancyhdr}
\usepackage{amsthm}
\usepackage{latexsym}
\usepackage{mathrsfs}
\usepackage{verbatim}
\usepackage{graphicx}
\usepackage{epstopdf}
\usepackage{epsfig}

\usepackage{breqn}
\usepackage{authblk}
\usepackage{setspace}

\usepackage{subfigure}
\usepackage{float}

\newtheorem{theorem}{Theorem}
\newtheorem{remark}{Remark}

\newtheorem{lemma}{Lemma}

\newtheorem{algorithm}{Algorithm}
\newtheorem{problem}{Problem}

\newcommand{\nmgrad}[1]{\|#1\|_{1,k}}
\newcommand{\nmzero}[1]{\|#1\|_{0,k}}
\newcommand{\nmcurl}[1]{\|#1\|_{{\rm curl},k}}
\newcommand{\nmdiv}[1]{\|#1\|_{{\rm div},k}} 

\begin{document}

\title{Stable Finite Element Methods Preserving
  $\nabla \cdot \bm{B}=0$ Exactly for MHD Models \thanks{This material
    is based upon work supported in part by the US Department of
    Energy Office of Science, Office of Advanced Scientific Computing
    Research, Applied Mathematics program under Award Number
    DE-SC-0006903 and by Beijing International Center for Mathematical
    Research of Peking University, China.}
}

 \author[1]{Kaibo Hu\thanks{hukaibo02@gmail.com} }
    \author[2]{Yicong Ma\thanks{yxm147@psu.edu} }
    \author[2]{Jinchao Xu\thanks{xu@math.psu.edu} }
    \affil[1]{Beijing
    International Center for Mathematical Research , Peking
    University, Beijing 100871, P. R. China}
    \affil[2]{Department of Mathematics,The Pennsylvania State
    University, University Park, PA 16802, USA}

\date{}

\maketitle

\begin{abstract}
This paper is devoted to the design and analysis of some
  structure-preserving finite element schemes for the
  magnetohydrodynamics (MHD) system.  The main feature of the method
  is that it naturally preserves the important Gauss law, namely
  {$\nabla\cdot\bm{B}=0$}.  In contrast to most existing
  approaches that eliminate the electrical field variable {$\bm{E}$} and
  give a direct discretization of the magnetic field, our new
  approach discretizes the electric field {$\bm{E}$} by
  N\'{e}d\'{e}lec type edge elements for {$H(\mathrm{curl})$},
  while the magnetic field {$\bm{B}$} by Raviart-Thomas
  type face elements for {$H(\mathrm{div})$}.  As a result, the
  divergence-free condition on the magnetic field holds exactly on the discrete
  level.

  For this new finite element method, an energy stability estimate can be
  naturally established in an analogous way as in the continuous case.
  Furthermore, well-posedness is rigorously established in the paper
  for both the Picard and Newton linearization of the fully nonlinear
  systems by using the Brezzi theory for both the continuous and
  discrete cases.  This well-posedness naturally leads to robust (and
  optimal) preconditioners for the linearized systems.

\end{abstract}

\section{Introduction}
Magnetohydrodynamics (MHD) studies the interaction of electromagnetic
fields and conducting fluids. Applications of MHD on different scales
can be found in many disciplines such as astrophysics, engineering related
to liquid metal, and controlled thermonuclear fusion.  There is vast
literature devoted to various aspects of MHD.  In this work, we
focus on an incompressible MHD model, and discuss an energetically stable 
mixed finite element discretization that preserves the divergence-free condition 
for the magnetic field.

We model the interaction of a fluid, with fluid velocity denoted by $\bm{u}$, fluid pressure $p$, an electric field $\bm{E}$, and a magnetic field $\bm{B}$. Given a bounded domain {$\Omega\subset \mathbb{R}^3$}, we consider the
following dimensionless MHD model for {$\bm{x}\in \Omega$} and {$t>0$}:
\begin{dgroup}[compact]
\label{eq:dimensionless}
\begin{dmath}
\frac{\partial \boldsymbol{u}}{\partial t} 
+ ( \boldsymbol{u} \cdot \nabla) \boldsymbol{u}
- \frac{1}{Re} \Delta \boldsymbol{u}
- S \boldsymbol{j} \times \boldsymbol{B}
+ \nabla p
= \boldsymbol{f} , 
\label{mom1} 
\end{dmath}
\begin{dmath}
\boldsymbol{j}
- \frac{1}{Rm} \nabla \times \boldsymbol{B}
= \boldsymbol{0} ,
\label{curl_B1} 
\end{dmath}
\begin{dmath}
\frac{\partial \boldsymbol{B}}{\partial t}
+ \nabla \times \boldsymbol{E} = \boldsymbol{0} ,
\label{induction1}
\end{dmath}
\begin{dmath}
\nabla \cdot \boldsymbol{B} = 0,
\label{div_B_free1} 
\end{dmath}
\begin{dmath}
\nabla \cdot \boldsymbol{u} = 0,
\label{div_u_free1} 
\end{dmath}
\end{dgroup}
where
\begin{equation}
\label{ohm1}
\bm{j}= \bm{E}+\bm{u}\times \bm{B}.
\end{equation}
and the coefficients are the fluid Reynolds number {$Re$}, magnetic Reynolds number {$Rm$}, and coupling number {$S$}. The initial conditions for the fluid velocity, magnetic field are given for {$\bm{x}\in\Omega$}:
\begin{eqnarray}
\label{u0}
&\bm{u}(\bm{x},0) =  \bm{u}_{0}(\bm{x}),\\
\label{B0}
& \bm{B}(\bm{x},0) =  \bm{B}_{0}(\bm{x}),
\end{eqnarray}
and the boundary conditions are given for {$\bm{x}\in\partial \Omega$} and {$t>0$}
\begin{align}
\label{bdry-u}
& \bm{u} = \bm{0}, \\
\label{bdry-B}
& \bm{B}\cdot \bm{n} = 0,\\
\label{bdry-E}
& \bm{E}\times \bm{n} = \bm{0}.
\end{align}
For the sake of simplicity, we assume that all parameters are positive constants
(since it is straightforward to generalize the main results in the
paper to the variable coefficient case).  The primary unknown physical variables in the
model are the velocity of fluid {$\bm{u}$}, the pressure {$p$} and the
magnetic field {$\bm{B}$}.  These quantities, once known, uniquely
determine the electric field {$\bm{E}$} and volume current density
{$\bm{j}$}.

This model is a combination of the incompressible Navier-Stokes equations and reduced Maxwell's equations, see \cite{Davidson.P.2001a} for more discussion on the model. The motion of fluid and electro-magnetic field is coupled by Lorentz force in equation \eqref{mom1}.

One major focus of this paper is the preservation of the Gauss law of
magnetic field \eqref{div_B_free1} on the discrete level (to simplify exposition, 
this condition is referred to as the divergence-free
condition below). The divergence-free condition
is a precise physical law in electro-magnetics, which plays a critical
role in the MHD system and its simulations. This condition means that
there is no source of the magnetic field in the domain. In other
words, it guarantees that no magnetic monopole exists. It is easy to see that 
{$\nabla\cdot \bm{B}=0$} is implied from \eqref{induction1}, provided
that the initial value {$\bm{B}_0$} in \eqref{B0} is
divergence-free.

Therefore, a basic assumption in this paper is that the initial data in \eqref{B0} satisfies
\begin{equation}
  \label{B0div0}
\nabla\cdot\bm{B}_0=0,\quad x\in\Omega. 
\end{equation}
Based on the previous argument, the divergence-free condition is automatically satisfied 
on the continuous level. However, this condition may no longer be true on the discrete level, 
if no special care is taken. The importance of the divergence-free condition on
discrete level in MHD simulations has been rigorous analyzed, for example, in
\cite{Brackbill.J;Barnes.D.1980a,Brackbill.J.1985a}. It has been
observed that small perturbations to this condition
can cause huge errors in numerical simulations of MHD 
\cite{Brackbill.J;Barnes.D.1980a,Brackbill.J.1985a,Evans.C;Hawley.J.1988a,Dai.W;Woodward.P.1998b,Toth.G.2000a}. A famous conclusion drawn by 
Brackbill and Barnes is that violation of the divergence-free condition 
on the discrete level will introduce a strong non-physical force 
\cite{Brackbill.J;Barnes.D.1980a}. This results in a significant
error in numerical simulation \cite{Dai.W;Woodward.P.1998b}.  In view
of energy conservation, violation of this condition leads to a
non-conservative energy integral \cite{Brackbill.J;Barnes.D.1980a},
$$
{ 1 \over 2 } \frac{d}{d t} \left( \|\bm{u}\|^{2} +
\frac{S}{Rm} \|\bm{B}\|^{2} \right) + \frac{1}{Re} \|\nabla \bm{u}\|^{2}
+ S \|\bm{j}\|^{2}
=(\bm{f}, \bm{u})-(\bm{B}\cdot \bm{u}, \nabla\cdot \bm{B}),
$$
and the error is proportional to the divergence of magnetic field. 
In contrast, the possible nonzero {$\nabla\cdot\bm{u}$} does not affect energy conservation, as 
long as the convection term is dealt with
appropriately.

Preserving the divergence-free condition on the discrete
level is a topic that has been extensively studied in literature. There are many
approaches to achieving this goal for various forms of MHD models
(such as ideal MHD, Hall MHD, resistive MHD), which can be classified mainly as divergence-cleaning methods, constrained transport methods, divergence-free bases and the like. An in-depth review on these methods can be found in \cite{Toth.G.2000a}.

The {\it Potential-based method} is widely used in simulations of MHD
system. One either introduces a vector potential to write {$\bm{B} = \nabla \times \bm{A}$} \cite{Jackson.J.1975a,Fautrelle.Y.1981a,Clarke.D;Norman.M;Burns.J.1986a,Shadid.J;Pawlowski.R;Banks.J;Chacon.L;Lin.P;Tuminaro.R.2010a,Jardin.S.2010a,Pekmen.B;Tezer-Sezgin.M.2013a,Hiptmair.R;Heumann.H;Mishra.S;Pagliantini.C.2014a}, or a scalar
potential to {$\bm{B}=-\nabla\psi$} and solves {$\Delta \psi=0$}
\cite{Bandaru.V;Boeck.T;Krasnov.D;Schumacher.J.1999a}.  For other
variations of this approach, refer \cite{Conraths.H.1996a}.

The {\it Lagrange multiplier method}, or `augmented' method, is also a popular approach. An additional term $\nabla r$ is introduced into to the induction equation, analogous to the pressure term in fluid momentum equation. In this approach, {$\bm{B}$} is weakly divergence-free \cite{Jiang.B;Wu.J;Povinelli.L.1996a,Demkowicz.L;Vardapetyan.L.1998a,Vardapetyan.L;Demkowicz.L.1999a,Schotzau.D.2004a,Salah.N;Soulaimani.A.1999a,Salah.N;Soulaimani.A;Habashi.W.2001a,Bris.C;Lelievre.T.2006a,Codina.R;Hernandez.N.2011a} (and reference therein).

The {\it Divergence-cleaning method} is another common strategy. The
central idea is to project the intermediate numerical solution 
{$\hat{\bm{B}}$ } to a divergence-free subspace by a linear
operator. One such method was first used by Brackbill and
Barnes in \cite{Brackbill.J;Barnes.D.1980a}, which is also referred
to as projection method. Later it was used in combination with
finite volume method \cite{Toth.G.2000a,Balsara.D;Kim.J.2004a}. Another method is the hyperbolic
divergence-cleaning method, which corrects the
divergence error by solving a hyperbolic equation 
\cite{Dedner.A;Kemm.F;Kroner.D;Munz.C;Schnitzer.T;Wessenberg.M.2002a}.
%

The {\it Constrained transport method} was first introduced by Evans and
Hawley \cite{Evans.C;Hawley.J.1988a} for the ideal MHD equations. It is
based on the Yee scheme \cite{K.S.Yee.1966a} for Maxwell's
equation. The motivation is to mimic the analytic fact that {$\mathrm{div}~ \mathrm{curl}~ \bm{u} = 0$} for arbitrary 
{$\bm{u} \in H(\mathrm{curl})$} on discrete level. For MHD systems, this
method is further developed by DeVore \cite{C.R.DeVore.1991a}, Dai et al. \cite{Dai.W;Woodward.P.1998a}, Ryu et al. \cite{Ryu.D;Miniati.F;Jones.T;Frank.A.1998a}, Liu et al. \cite{Liu.J;Wang.W.2001a}, Balsara et
al. \cite{Balsara.D;Spicer.D.1999a,Balsara.D.2004a}, Fey et al. \cite{Fey.M;Torrilhon.M.2003a},Londrillo et al. \cite{Londrillo.P;Zanna.L.2004a}, Rossmanith \cite{Rossmanith.J.2006a}, Helzel et al. \cite{Helzel.C;Rossmanith.J;Taetz.B.2011a}. A comparison between divergence cleaning and this method is provided in \cite{Balsara.D;Kim.J.2004a}.

{\it Divergence-free bases} are another means of satisfying the
divergence-free condition. The variables are discretized by divergence-free basis 
\cite{Ye.X;Hall.C.1997a,Zhang.S.2012a} as in the Stokes (Navier-Stokes) equation, and
\cite{Cai.W;Wu.J;Xin.J.2013a} for MHD equation. This idea can be also
used in combination with a discontinuous Galerkin (DG) method 
\cite{Cockburn.B;Li.F;Shu.C.2004a,Li.F;Shu.C.2005a,Li.F;Xu.L.2012a,Yakovlev.S;Xu.L;Li.F.2013a}.

{\it Other methods} are also used to preserve the divergence-free
condition. For example, the 8-wave formulation of MHD equations \cite{Powell.K.1997a} or 
methods relying on the original mathematical structure of the equations
\cite{Gunzburger.M;Meir.A;Peterson.J.1991a}. More techniques used on the conservation law can be found in
\cite{Toth.G.2000a}.

The discretization we adopt in this paper is based on finite
element methods.  There has been a lot of research on finite element
methods for MHD systems, for example, \cite{Gunzburger.M;Meir.A;Peterson.J.1991a,Armero.F;Simo.J.1996a,Wiedmer.M.2000a,Gerbeau.J;Lelievre.T;Bris.C.2003a,Guermond.J;Minev.P.2003a,Schneebeli.A;Schotzau.D.2003a,Hasler.U;Schneebeli.A;Schotzau.D.2004a,Schotzau.D.2004a,Houston.P;Schoetzau.D;Wei.X.2008a,Prohl.A.2008a,Aydn.S.2010a,Banas.L;Prohl.A.2010a,Shadid.J;Pawlowski.R;Banks.J;Chacon.L;Lin.P;Tuminaro.R.2010a,Badia.S;Codina.R;Planas.R.2013a}.  In view of the Sobolev spaces used for the magnetic field
variable {$\bm{B}$}, existing finite element methods can be roughly
classified into two different categories:  the first uses {$H(\mathrm{grad})$}, and the second
uses the {$H(\mathrm{curl})$} space.

For methods based on {$H^{1}(\Omega)$}, we refer to
\cite{Gunzburger.M;Meir.A;Peterson.J.1991a,Armero.F;Simo.J.1996a,Wiedmer.M.2000a,Gerbeau.J;Lelievre.T;Bris.C.2003a,Houston.P;Schoetzau.D;Wei.X.2008a,Aydn.S.2010a}.  Most of these
methods only preserve the divergence-free of {$\bm{B}$} in
a weak sense, for example, Sch\"{o}tau
\cite{Schotzau.D.2004a}, by adding a Lagrange multiplier {$r\in
  {H}^{1}(\Omega)$} to numerical formulation. The convergence of
some of these methods are not guaranteed on non-smooth
concave domains \cite{Ida.N;Bastos.J.1997a,Monk.P.2003a}. 
For discretization of {$H(\mathrm{curl})$} for 
{$\bm{B}$}, we refer to \cite{Guermond.J;Minev.P.2003a,Schneebeli.A;Schotzau.D.2003a,Hasler.U;Schneebeli.A;Schotzau.D.2004a,Schotzau.D.2004a,Prohl.A.2008a,Banas.L;Prohl.A.2010a,Badia.S;Codina.R;Planas.R.2013a}.  Again, these discretizations
only assure the divergence-free condition for {$\bm{B}$} in
a weak sense.

In view of the Sobolev spaces used for the magnetic field variable
{$\bm{B}$}, we use {$H(\mathrm{div})$} as a basis in the
finite element discretization scheme studied in our paper. To accomplish this,
we discretize the electric-field variable $\bm{E}$ in
{$H(\mathrm{curl})$}.  Similar to Maxwell equations, we
view the electric field {$\bm{E}$} as 1-form and magnetic field
{$\bm{B}$} as 2-form and discretize these two variables 
by the corresponding discrete 1-form and 2-form.  More specifically, we use a mixed finite
element formulation that discretizes {$\bm{E}$} in 
{$H(\mathrm{curl})$} by N\'{e}d\'{e}lec elements \cite{Nedelec.J.1980a,Nedelec.J.1986a}, and discretizes
{$\bm{B}$} in {$H(\mathrm{div})$} by Raviart-Thomas
elements \cite{Raviart.P;Thomas.J.1977a}.

In our new discretization of the MHD system, Faraday's law still
holds exactly on the discrete level and, as a result, the Gauss law 
is automatically satisfied.
Thanks to the exact preservation of both Faraday and Gauss laws, our
discrete finite element schemes also has many desirable mathematical
properties (such as energy estimates) and well-posedness
(existence, uniqueness, and stability) can be rigorously established by
using the classic theory of Brezzi \cite{Brezzi.F.1974a,Boffi.D;Brezzi.F;Fortin.M.2013a} for the mixed finite element
method.  

In our formulation and analysis, we make use of well-established
numerical techniques and relevant mathematical theories for solving
Maxwell equations that are based on discrete differential forms or
finite element exterior calculus, see work of Bossavit \cite{Bossavit.A.1998a,Bossavit.A.2005a}, 
Hiptmair \cite{Hiptmair.R.2002a}, 
Arnold et al. \cite{Arnold.D;Falk.R;Winther.R.2006a,Arnold.D;Falk.R;Winther.R.2010a}.
In particular, we employ a mixed finite element formulation for the
Hodge Laplacian in an abstract framework studied by Arnold, Falk and Winther in
\cite{Arnold.D;Falk.R;Winther.R.2006a,Arnold.D;Falk.R;Winther.R.2010a}.

The rest of the paper is organized as follows. Notations and
introduction to finite element spaces are given in \S
\ref{sec:preliminary}. A new variational formulation and energy
estimates (both continuous and discrete cases) are presented in \S
\ref{sec:variation}. Linearized discrete schemes based on Picard
and Newton iterations are analyzed in \S \ref{sec:linearization}. We close in \S \ref{sec:concluding} 
with some concluding remarks.

\section{Preliminaries}\label{sec:preliminary}
As mentioned above, we consider our MHD model in an open bounded domain
{$\Omega\subset \mathbb{R}^3$}.  We assume that {$\Omega$} has a Lipschitz
continuous boundary. We remark that {$\Omega$} is not assumed to be 
convex. For simplicity of exposition, we assume that {$\Omega$} is a
simply connected polygon in the rest of the paper.

\subsection{Sobolev spaces}
We briefly introduce notation for some standard Sobolev
spaces.  First, the {$L^{2}$} inner product and norm are denoted by 
{$(\cdot, \cdot)$} and {$\|\cdot\|$} respectively:
$$
(u,v):=\int_{\Omega}u\cdot v \mathrm{d}x,\quad
\|u\|:=\left(\int_{\Omega} \lvert u\rvert^2 \mathrm{d}x\right)^{1/2}. 
$$
With a slight abuse of notation, {$L^2(\Omega)$} will be used to denote
both the scalar and vector {$L^2$} spaces.  Given a linear operator {$D$}, we define:
$$
H(D,\Omega):=\{v\in L^2(\Omega), Dv\in L^2(\Omega)\}  
$$
and
$$
H_0(D,\Omega):=\{v\in H(D, \Omega), t_{D}v=0 \mbox{ on } \partial\Omega\}. 
$$
Here, $t_{D}$ is the trace operator:
$$
t_{D}v=
\left\{
  \begin{array}{cc}
    v & D=\mathrm{grad},\\
    v\times n & D=\mathrm{curl},\\
    v\cdot n & D=\mathrm{div}.
  \end{array}
\right.
$$
We note that {$L^2(\Omega)$} can be viewed as {$H(id,\Omega)$}. And we
often use the following notation:
$$
L^2_0(\Omega)=\left\{v\in L^2(\Omega):  \int_\Omega v=0 \right\}.
$$
When {$D=\mathrm{grad}$}, we often use the notation:
$$
H^1(\Omega)=H(\mathrm{grad}, \Omega), \quad
H^1_0(\Omega)=H_0(\mathrm{grad}, \Omega).
$$
We further define:
$$
H_0(D0,\Omega):=\{v\in H_0(D, \Omega), Dv=0\},
$$
in particular, 
$$
H_0(\mathrm{div}0,\Omega):=\{\bm{v}\in H_0(\mathrm{div}, \Omega), \nabla\cdot \bm{v}=0\}.
$$
We use the space {$L^{p}$} and {$H^{-1}$} with norms denoted by
$$
\|{v}\|_{0,p}= \left( \int_\Omega|v|^p \right)^{1/p},\quad \|{v}\|_{-1}=\sup_{{\phi}\in
  H^1_0(\Omega)}\frac{\langle v,\phi \rangle}{\|\nabla\phi\|} 
$$
and also the following space
$$
L^2([0,T]; H^{-1}(\Omega))= \left \{f:
\int_{0}^T\|f(t,\cdot)\|_{-1}^2dt<\infty \right \}. 
$$
We will make the following assumption for the data throughout the paper:
\begin{equation}
  \label{data}
\bm{u}^0, \bm{B}^0\in L^2(\Omega), \quad \bm{f}\in L^2([0,T]; H^{-1}(\Omega)).
\end{equation}

\subsection{Finite element spaces}
With the notation introduced in the previous sections, we use the
following Sobolev spaces for the physical variables:
$$
(\bm{u}, p)\in H_{0}^{1}(\Omega)^{3}\times L^2_0(\Omega)
\quad \mbox{ and } \quad
(\bm{E}, \bm{B})\in {H}_{0}(\mathrm{curl}; \Omega)\times
{H}_{0}(\mathrm{div}; \Omega).  
$$
We use familiar finite element spaces to discretize the above
variables as described below.

\subsubsection{Finite element spaces for {$(\bm{E}, \bm{B})\in
  {H}_{0}(\mathrm{curl}; \Omega)\times {H}_{0}(\mathrm{div};
  \Omega)$}.}
We use the well-studied finite element spaces, namely the N\'{e}d\'{e}lec
edge elements and Raviart-Thomas face elements (and their
generalizations) for {$H_{0}(\mathrm{curl}; \Omega)$} and
{$H_{0}(\mathrm{div}; \Omega)$} respectively.  There is now a unified
theory for these types of elements, see
\cite{Hiptmair.R.2002a,Arnold.D;Falk.R;Winther.R.2006a}.  These finite element spaces are
best described in terms of the discrete de Rham complex. 
Figure~\ref{exact-sequence} illustrates the exact sequences on both continuous
and discrete levels and, in Figure~\ref{deRham}, the degrees of
freedom (DOF) of one family of the finite elements of lowest
order.

\begin{figure}[ht!]
\begin{equation*}
\begin{CD}
H_0(\mathrm{grad})   @> {\mathrm{grad}} >> H_0(\mathrm{curl}) 
@>{\mathrm{curl}} >> H_0(\mathrm{div})  @> {\mathrm{div}} >> L_0^2  \\    
 @VV\Pi_h^{\mathrm{grad}} V @VV\Pi_h^{\mathrm{curl}} V @VV\Pi_h^{\mathrm{div}} V
@VV\Pi_h^0 V\\  
H^h_0(\mathrm{grad})  @>{\mathrm{grad}} >> H^h_0(\mathrm{curl})  @>{\mathrm{curl}} >> H^h_0(\mathrm{div}) @> {\mathrm{div}} >> L^{2,h}_0
\end{CD}
\end{equation*}
\caption{Continuous and discrete de Rham sequence}
\label{exact-sequence}
\end{figure}
\begin{figure}[ht!]
\begin{center}
\includegraphics[width=.8in]{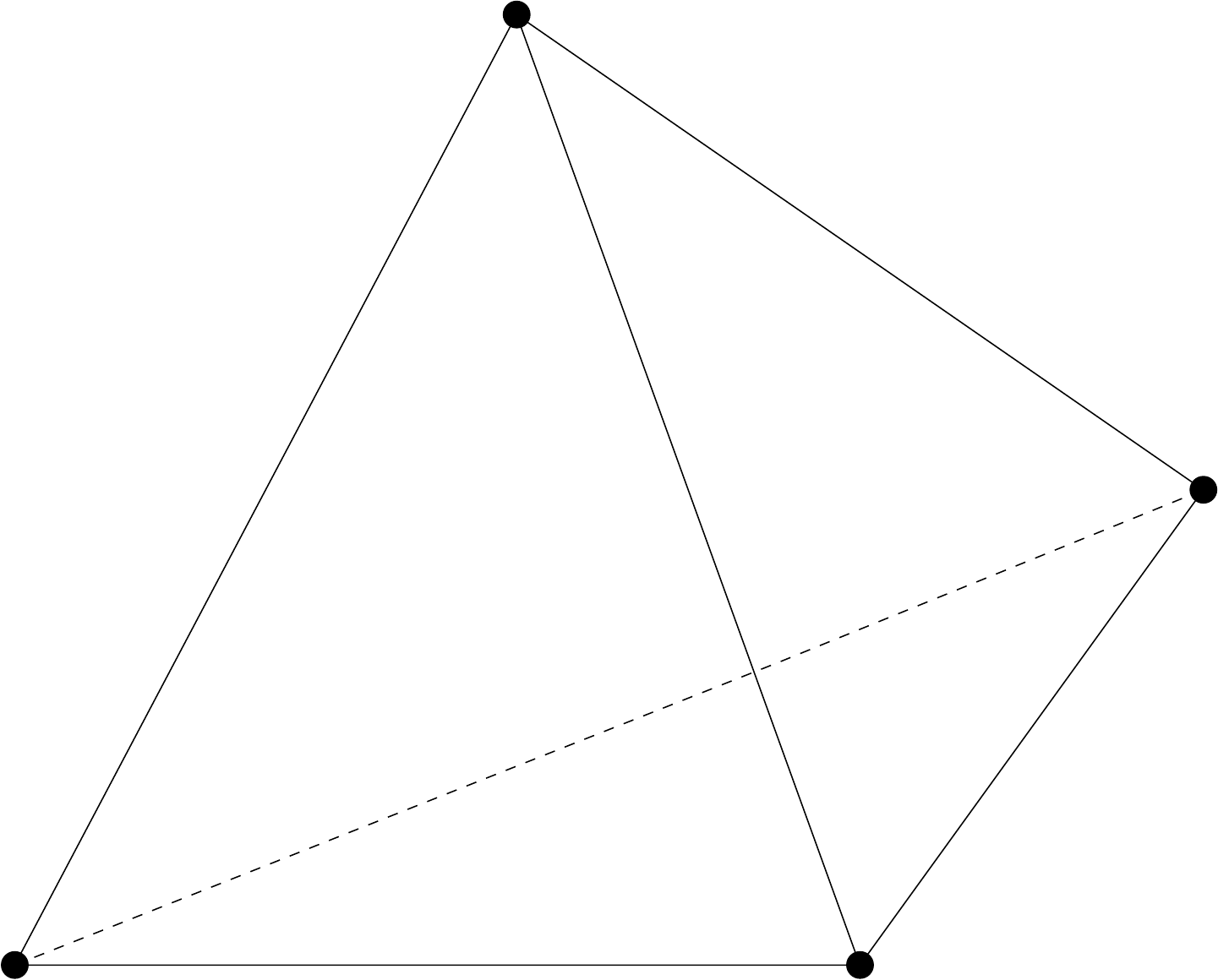}
\includegraphics[width=.8in]{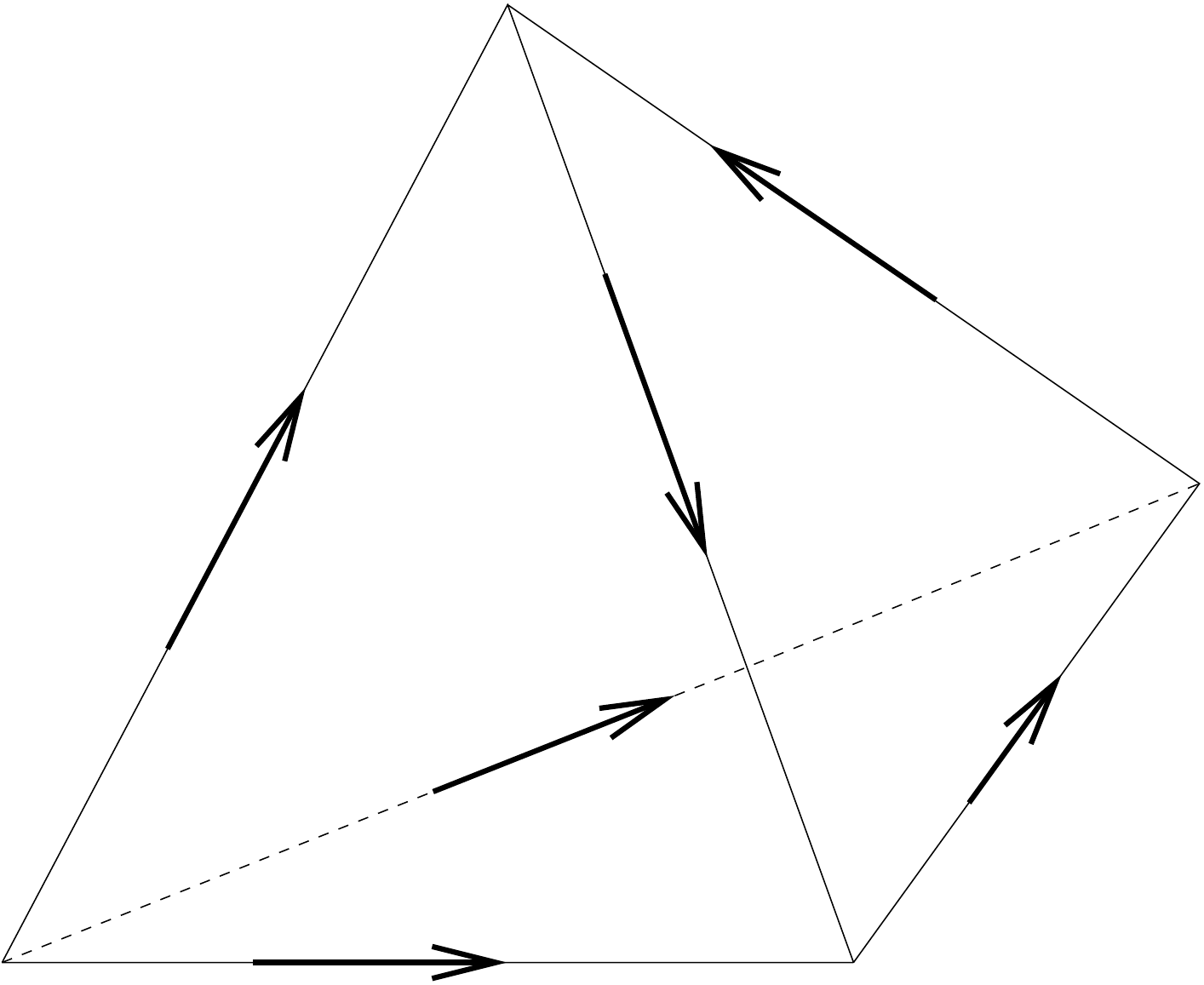}
\includegraphics[width=.8in]{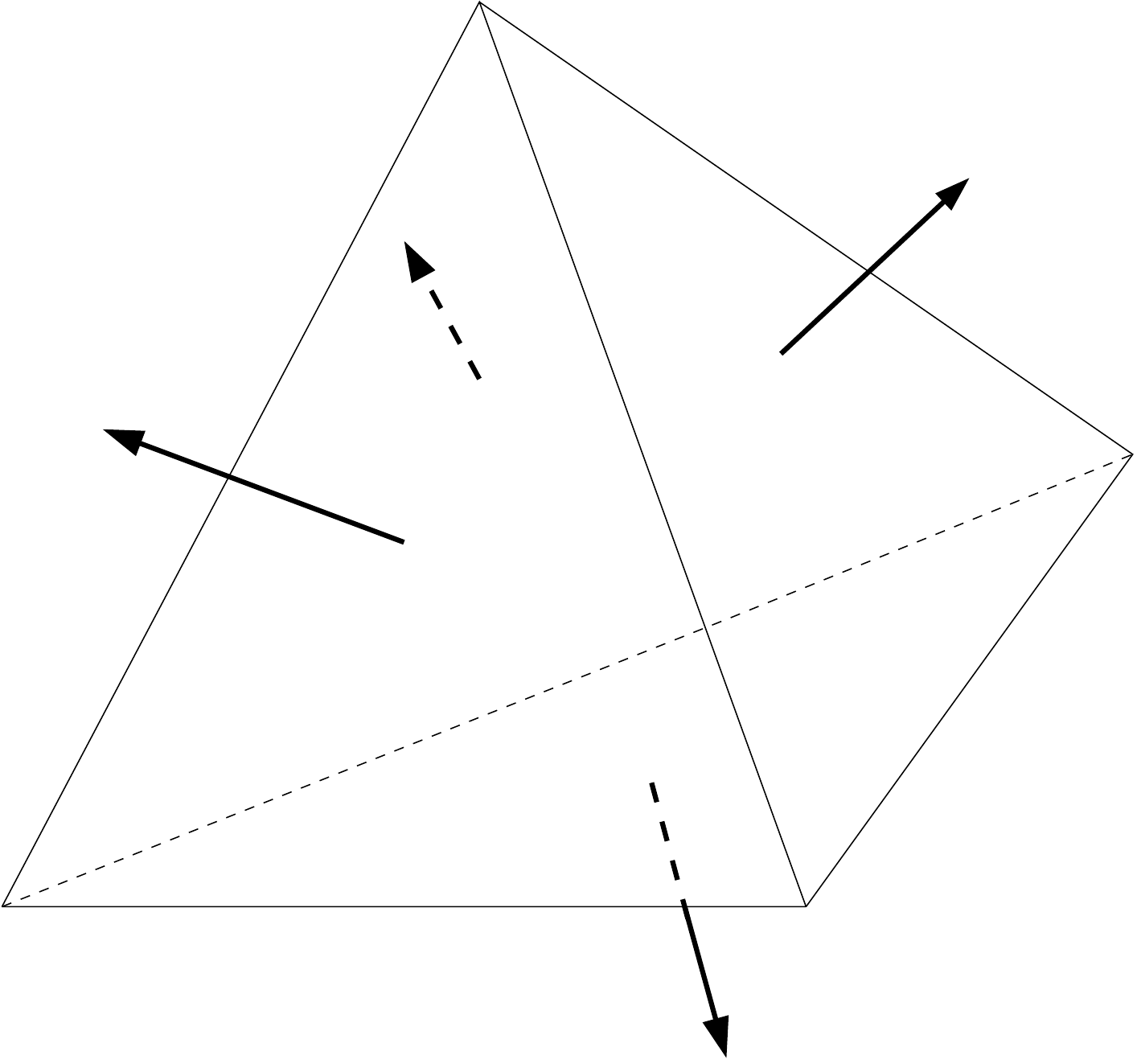}
\includegraphics[width=.8in]{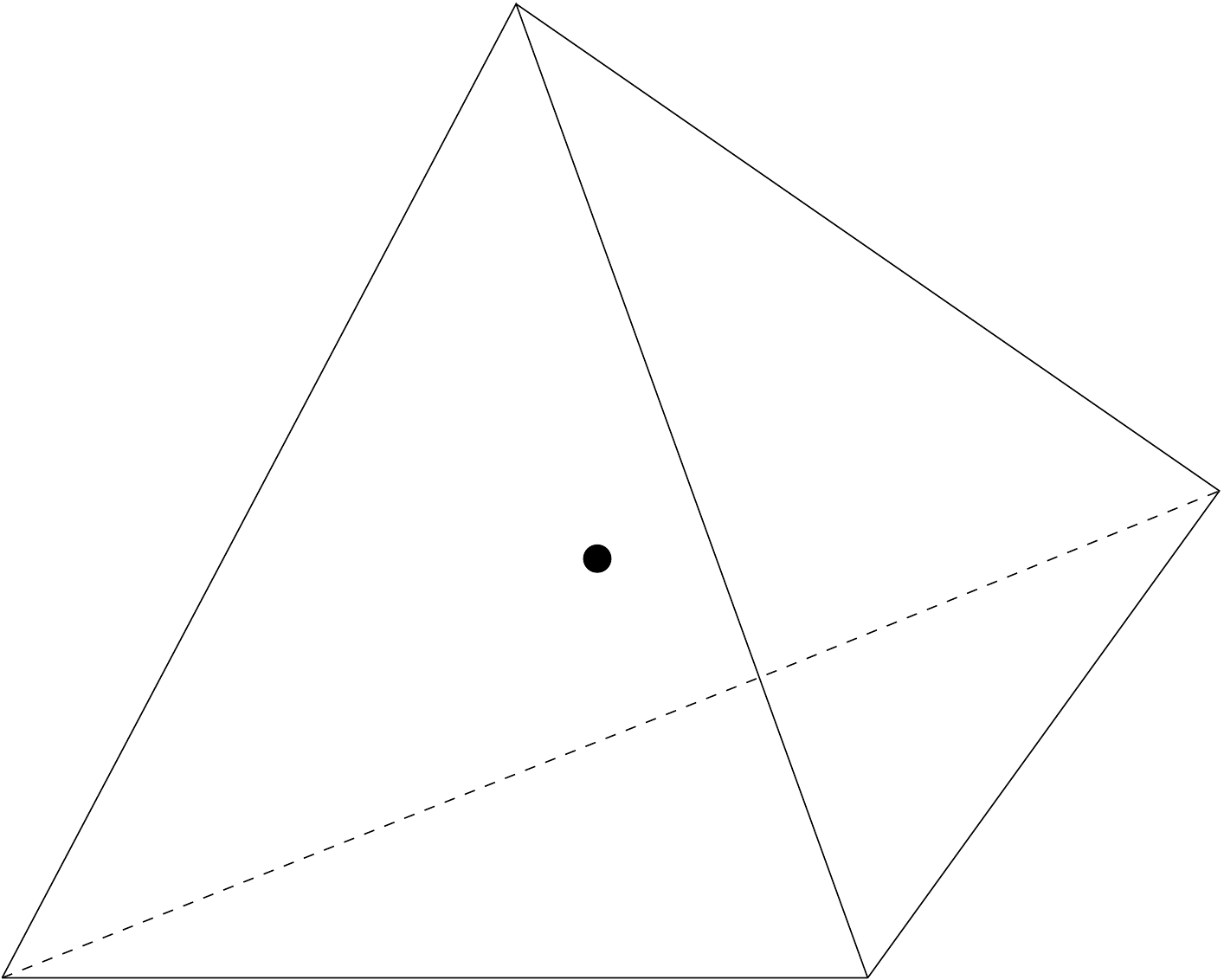}
\caption{DOF of finite element de Rham sequence of lowest order}
\label{deRham}
\end{center}
\end{figure}

We will use {$\bm{V}^{c}$} and {$\bm{V}^{d}$} to denote {$H_{0}(\mathrm{curl};\Omega)$} and {$H_{0}(\mathrm{div};\Omega)$}:
$$
\bm{V}^{c}:= H_0(\mathrm{curl}; \Omega),\quad   \bm{V}^{d} := H_0(\mathrm{div}; \Omega),
$$
and {$\bm{V}^{c}_{h}$}, {$\bm{V}_{h}^{d}$} for their finite element
subspaces as shown in Figure~\ref{exact-sequence}, namely
$$
\bm{V}^{c}_{h} \times \bm{V}^{d}_{h}:= H^h_0(\mathrm{curl}) \times
H^h_0(\mathrm{div})\subset \bm{V}^{c}\times \bm{V}^{d}.
$$

Because of the Gauss law \eqref{div_B_free1}, {$H_{0}(\mathrm{div},
\Omega)$} functions with vanishing divergence play an important role in
the analysis. Hence we define
$$
\bm{V}^{d,0}:= H_{0}(\mathrm{div}0, \Omega)=
\{\bm{C}\in \bm{V}^d: \nabla\cdot\bm{C}=0\}
$$
and the discrete space:
$$
\bm{V}^{d,0}_{h}:=H^h_0(\mathrm{div}0, \Omega)= \{\bm{C}\in \bm{V}^{d}_{h}: \nabla\cdot \bm{C}=0  \}.
$$

%

\subsubsection{Stable Stokes pairs for {$(\bm{u}, p)\in
  H_{0}^{1}(\Omega)^{3}\times\in L^2_0(\Omega)$} } 
We use {$\bm{V}_h$} to denote the finite element subspace of
{$H_{0}^{1}(\Omega)^{3}$} and {$Q_h$} for subspace of {$L_{0}^{2}(\Omega)$}.
The basic requirement of this pair of finite element spaces is that
they satisfy the following inf-sup conditions:
\begin{equation}
  \label{inf-sup}
\inf_{q_h\in  Q_h}\sup_{\bm{v}_h\in \bm{V}_h}\frac{(\nabla\cdot \bm{v}_h,q_h)}{\|\nabla \bm{v}_h\|\;\|q_h\|}\ge \beta>0,
\end{equation}
for some positive constant {$\beta$} that is independent of {$h$}. Many existing pairs of stable Stokes elements can be used, like
Taylor-Hood elements \cite{Girault.V;Raviart.P.1986a,Boffi.D;Brezzi.F;Fortin.M.2013a}. 

We define
\begin{equation}
  \label{spaceV}
\bm{V}= H_{0}^{1}(\Omega)^{3}, \quad Q=L_{0}^{2}(\Omega),  
\end{equation}

\begin{equation}
  \label{spaceV0}
\bm{V}^0=\{\bm{v}\in \bm{V}: (\nabla \cdot \bm{v},q)=0,\quad\forall
q\in L^{2}_{0}(\Omega)\}
\end{equation}
and 
\begin{equation}
  \label{spaceVh0}
\bm{V}_h^0=\{\bm{v}_h\in \bm{V}_h: (\nabla \cdot \bm{v}_h,q_h)=0,\quad\forall
q_h\in Q_h\}. 
\end{equation}

We note that 
\begin{equation}
  \label{spaceV0}
\bm{V}^0=\{\bm{v}\in \bm{V}: \nabla \cdot \bm{v}=0\},
\end{equation}
but on the discrete level
\begin{equation}
  \label{spaceVh0}
\bm{V}_h^0=\{\bm{v}_h\in \bm{V}_h: P_h\nabla \cdot \bm{v}_h=0\} 
\end{equation}
where {$P_h: L^2(\Omega)\mapsto Q_h$} is the projection, defined as follows:
\begin{equation}
  \label{L2projection}
(P_hq,\phi_h)=(q,\phi_h), \quad\forall q\in L^2(\Omega), \phi_h\in Q_h.  
\end{equation}
Namely {$\bm{V}_h^0\not\subset \bm{V}^0$} in general unless
{$\nabla\cdot\bm{V}_h\subset Q_h$}.  

\begin{remark}
Sometimes, we may use all four spaces in one of the
  exact sequences, for example, the second or third sequence in the four resolutions of {$\mathbb{R}^{3}$} to discrete all the four variables $(\bm{u},\bm{E},\bm{B}, p)$.

  But, in general, $\bm{V}_h$ is not necessary the same as the
  discrete $H^h_0(grad;\Omega)$ nor $Q_h$ is necessarily the same as
  the discrete space $L^2_{0,h}$ as shown in
  Figure~\ref{exact-sequence}.

  There is another interesting use of the spaces in
  Figure~\ref{exact-sequence}.  As studied in
  \cite{Cockburn.B;Li.F;Shu.C.2004a,Dios.B;Brezzi.F;Marini.L;Xu.J;Zikatanov.L.2014a} for a DG
  formulation. One could take the following Stokes pair from
  Figure~\ref{exact-sequence} as
$$
\bm{V}_h= H^h_0(\mathrm{div}) \mbox{ and } Q_h=L^2_{0,h}.
$$

But we will not discuss this choice in detail in this paper. 
\end{remark}


\subsection{On the boundary conditions for the MHD model}
The boundary conditions \eqref{bdry-u}-\eqref{bdry-E} call for some
explanations due to the fact that the MHD model only employs some, but not all, of 
Maxwell's equations.  The boundary condition
\eqref{bdry-u} for {$\bm{u}$} is standard for NS equation. For simplicity, the pure
Dirichlet boundary condition is considered here.  The boundary
conditions \eqref{bdry-B} and \eqref{bdry-E} are actually not 
independent.  More precisely, the boundary condition \eqref{bdry-E} implies 
\eqref{bdry-B} whenever the initial condition satisfies
$\bm{B}_{0}(x)\cdot \bm{n} = 0$, for any $x\in\Omega $. 
This is due to the relation $\mathrm{div}_{\partial \Omega}( \bm{E}\times \bm{n} )=(\nabla\times
  \bm{E})\cdot \bm{n}|_{\partial \Omega}$, see \cite{Monk.P.2003a,Boffi.D;Brezzi.F;Fortin.M.2013a}.

\section{Variational formulation and finite element
  discretizations}\label{sec:variation}
The first task in designing our new method is to introduce an
appropriate variational formulation for \eqref{mom1}-\eqref{div_u_free1} with given boundary conditions and initial data.

\subsection{Basic spaces and weighted norms}
We observe that it is convenient to group the variables
{$(\bm{u}, \bm{E},\bm{B})$} to form the following mixed pair of Sobolev
spaces:
\begin{equation}
\label{space}
\bm{X}:= \bm{V}\times \bm{V}^{c} \times
\bm{V}^{d} 
\quad \mbox{and} \quad Q=L_0^2(\Omega),
\end{equation}

and the corresponding finite element spaces:
\begin{equation}
  \label{space-FE}
\bm{X}_h:= \bm{V}_{h}\times \bm{V}_{h}^{c}  \times  \bm{V}_{h}^{d}
\quad \mbox{and} \quad Q_h.
\end{equation}

We also use the following subspaces
\begin{equation}
\bm{X}^{\bm{B},0}:= \bm{V}\times \bm{V}^{c}\times
\bm{V}^{d,0}, \quad
 \bm{X}^{\bm{u},0}:= \bm{V}^0\times \bm{V}^{c}\times
\bm{V}^{d},
\end{equation}
and the corresponding finite element spaces
\begin{equation}
\bm{X}^{\bm{B},0}_h:=  \bm{V}_h\times \bm{V}_{h}^{c} \times
\bm{V}_{h}^{d,0}, \quad
\bm{X}^{\bm{u}, 0}_h:= \bm{V}_h^0\times \bm{V}_{h}^{c}\times
\bm{V}_{h}^{d}.
\end{equation}

For a unified presentation for both the continuous and discrete
formulations, we also use the same notation {$\bm{X}$}, {$Q$} and
{$\bm{X}^{\bm{u},0}$ }, {$\bm{X}^{\bm{B},0}$} to denote the corresponding finite element
spaces:
\begin{equation}
  \label{space-h}
\bm{X}=\bm{X}_{h} \quad \mbox{and} \quad Q=Q_h,
\end{equation} 
and 
\begin{equation}
  \label{space0-h}
\bm{X}^{\bm{u},0} =\bm{X}^{\bm{u},0}_h 
\quad \mbox{and} \quad 
\bm{X}^{\bm{B},0}=\bm{X}^{\bm{B},0}_h. 
\end{equation}

Let {$k$} be a positive number that denotes the time-step size. 
We introduce the following weighted Sobolev norms for both the
continuous Sobolev spaces and the corresponding finite element spaces:
for any
{$(\bm{E},\bm{B})\in \bm{V}^{c}\times \bm{V}^{d}$}, we define
$$
\nmcurl{\bm{E}}^{2}:=\|\bm{E}\|^{2}+k \|\nabla\times \bm{E} \|^{2}
\quad \mbox{and} \quad
\nmdiv{\bm{B}}^{2}:={k^{-1}}\|\bm{B}\|^{2}+\|\nabla\cdot \bm{B}\|^{2}.
$$

For any {$(\bm{u},p)\in \bm{V}\times Q$}, we define 
$$
\nmgrad{\bm{u}}^{2}:={k^{-1}}\|\bm{u}\|^{2}+ \|\nabla \bm{u}
\|^{2}+{k^{-1}}\|\mathbb{P}\nabla\cdot \bm{u}\|^{2} 
\quad \mbox{and} \quad
\nmzero{p}^{2}:=k \|p\|^{2}.
$$

When {$\bm{u}\in \bm{V}_h$}, {$\mathbb{P}={P}_h$} is the {$L^2$}-projection
defined by \eqref{L2projection}.
Otherwise, we take {$\mathbb{P}$} to be the identity operator. 
 
The norms of {$\bm{X}$} and {$Q$} are defined in a standard way:
$$
\|(\bm{v},\bm{F}, \bm{C})\|_{\bm{X}}^{2}:= \|\bm{v}\|_{1,k}^{2}+\|\bm{F}\|_{\mathrm{curl},k}^{2}+\|\bm{C}\|_{\mathrm{div},k}^{2}, \quad
 \|q\|_{Q}=\|q\|_{0,k}.
$$

Correspondingly, the term {$k^{-1}\|\mathbb{P}\nabla
\cdot \bm{v}\|$} vanishes in the norm of {$\bm{X}^{\bm{u},0}$}; 
and in norm of {$\bm{X}^{\bm{B},0}$}, we have 
{$\|\nabla\cdot \bm{B}\|=0$}. Dual norms are defined as
$$
\|\bm{h}\|_{\bm{X}^{\ast}}:= \sup_{\bm{\eta}\in \bm{X}}\frac{\langle
  \bm{h} , \bm{\eta}\rangle}{\|\bm{\eta}\|_{\bm{X}}}, \quad
\|g\|_{Q^{\ast}}:= \sup_{q\in Q}\frac{ \langle g, q \rangle
}{\|q\|_{Q}}.
$$

\subsection{Variational formulation and finite element discretizations}
There are many different possible variational formulations for the MHD
models.  A prominent feature of the formulation used in this work is that 
the electric field is kept, while the divergence-free condition on the magnetic 
field \eqref{div_B_free1} is not explicitly enforced. In many existing
discretizations, the electric field is
eliminated, and a Lagrangian multiplier is
introduced to preserve the divergence-free condition of {$\bm{B}$} (in
a weak sense).  By maintaining the electric field {$\bm E$} as an independent
variable, the divergence-free condition
\eqref{div_B_free1} is satisfied naturally and precisely on both
the continuous and discrete levels.

We have four independent physical variables in our formulation,
namely, the fluid velocity $\bm{u}$, the fluid pressure $p$, the electric field 
$\bm{E}$, and the magnetic field $\bm{B}$.

The variational formulation used for both continuous
\eqref{space} and discrete levels \eqref{space-h} is as follow.
\begin{problem}\label{prob:semi}
Find {$(\bm{u}, \bm{E}, \bm{B})\in \bm{X}$} and {$ p\in Q$} such that for
any {$(\bm{v}, \bm{F}, \bm{C})\in \bm{X} $} and {$q\in Q$},
\begin{align}
\nonumber
&  \left( \frac{\partial \bm{u}}{\partial t},\bm{v}  \right) +
  \frac{1}{2} \left[ \left( \bm{u}\cdot \nabla \bm{u}, \bm{v} \right) - ( \bm{u}\cdot
  \nabla \bm{v}, \bm{u}) \right]  
  + \frac{1}{Re} (\nabla \bm{u}, \nabla \bm{v}) \\ \label{mom2}
& \quad \quad
  - S (\bm{j}\times \bm{B},\bm{v} ) - (p,\nabla\cdot \bm{v}) 
 = (\bm{f},\bm{v}),\\
\label{curl_B2} 
& (\bm{j},\bm{F}) - \frac{1}{Rm} ( \bm{B}, \nabla\times \bm{F}) =0, \\
\label{induction2}
& \left( \frac{\partial \bm{B}}{\partial t}, \bm{C} \right) 
+ (\nabla\times \bm{E}, \bm{C}) = 0, \\ 
\label{div_u_free2} 
& (\nabla\cdot \bm{u}, q) =0,
\end{align}
where {$\bm{j}$} is given by Ohm's law: 
{$\bm{j} = \bm{E} + \bm{u}\times \bm{B} $}.
\end{problem}

We use the implicit Euler scheme to discrete the time variable in
Problem \ref{prob:semi} and to obtain the following approximation
of the MHD model:
\begin{problem}\label{prob:nonlinear}  
Given {$(\bm{u}^0, \bm{B}^0)$}, for {$n=1,2,3,\ldots$}, find
  {$(\bm{u}^n, \bm{E}^{n}, \bm{B}^n)\in \bm{X}$} and {$ p^n\in Q$} such
  that for any {$(\bm{v}, \bm{F}, \bm{C})\in X $} and {$q\in Q$},
\begin{dgroup*}
\begin{dmath}
\left( \frac{\bm{u}^{n}-\bm{u}^{n-1}}{k},\bm{v} \right) 
+  {1 \over 2}[( \bm{u}^{n}\cdot \nabla \bm{u}^{n},\bm{v})
- ( \bm{u}^{n}\cdot \nabla\bm{v} , \bm{u}^{n})]
+ \frac{1}{Re} (\nabla \bm{u}^{n}, \nabla \bm{v})
- S (\bm{j}^{n}\times \bm{B}^{n},\bm{v} )
- (p^{n},\nabla\cdot \bm{v}) = (\bm{f}^{n},\bm{v}),
\label{nonlinear1} 
\end{dmath}
\begin{dmath}
(\bm{j}^{n},\bm{F}) - \frac{1}{Rm} ( \bm{B}^{n}, \nabla\times \bm{F}) =0, 
\label{nonlinear2} 
\end{dmath}
\begin{dmath}
\left( \frac{\bm{B}^{n}-\bm{B}^{n-1}}{k}, \bm{C} \right) 
+ (\nabla\times \bm{E}^{n}, \bm{C}) = 0, 
\label{nonlinear3}
\end{dmath}
\begin{dmath}
\left( \nabla\cdot \bm{u}^{n}, q \right) = 0, \label{nonlinear4} 
\end{dmath}
\end{dgroup*}
where 
{$\bm{j}^{n} = \bm{E}^{n}+\bm{u}^{n}\times \bm{B}^{n} $}.
\end{problem}

We note that the special treatment of the nonlinear convection term in
\eqref{mom2} and \eqref{nonlinear1} is based on the following identity that holds for
{$\nabla\cdot\bm{u}=0$},
\begin{equation}\label{modified-convection}
( \bm{u}\cdot \nabla \bm{u}, \bm{v})
=\frac{1}{2}[( \bm{u}\cdot \nabla \bm{u}, \bm{v}) 
- ( \bm{u} \cdot \nabla
\bm{v}, \bm{u})],\quad\forall \bm{v}\in V.
\end{equation}
This is a classical stabilization technique (c.f. \cite{Schotzau.D.2004a})
when {$\bm{X}\times Q$} is given by \eqref{space-h} and
{$\nabla\cdot \bm{V}\not\subset Q$}. 

\begin{remark} We note that 
  \begin{enumerate}
  \item When {$\bm{X}\times Q$} is given by \eqref{space}, Problem
    \ref{prob:semi} gives a variational formulation of the MHD
    model \eqref{mom1}-\eqref{div_u_free1} and Problem
    \ref{prob:nonlinear} gives a semi-discretization (in time) for the
    MHD model \eqref{mom1}-\eqref{div_u_free1}.

  \item When {$\bm{X}\times Q$} is given by \eqref{space-h}, Problem
    \ref{prob:semi} gives a semi-discretization (in space)
   and Problem \ref{prob:nonlinear} gives a
    full discretization for Problem~\ref{prob:semi}.
  \end{enumerate}
\end{remark}

When considering the spatially discrete problem given by \eqref{space-h}, 
we need to specify {$(\bm{u}^0, \bm{B}^0)\in \bm{X}_h$}. Naturally, the choice for 
{$(\bm{u}^0, \bm{B}^0)\in \bm{X}_h$} should be a good approximation of the
continuous initial date as given in \eqref{u0}
and \eqref{B0}.  While the choice of {$(\bm{u}^0, \bm{B}^0)\in
\bm{X}_h$} is not unique, we require that the following condition be
satisfied:
\begin{equation}
  \label{B0hdiv0}
\nabla\cdot\bm{B}^0=0,\quad \bm{x}\in\Omega. 
\end{equation}

One way to assure this condition is to use the interpolation operator
{$\Pi_h^{\rm div}$} as shown in Figure \ref{exact-sequence}: 
$\bm{B}^0= \Pi_h^{\rm div}\bm{B}_0$.
Thanks to the commutative diagram illustrated in Figure
\ref{exact-sequence}, {$\bm{B}^0$} satisfies
\eqref{B0hdiv0} because of \eqref{B0div0} and 
$$
\nabla\cdot\bm{B}^0= \nabla\cdot (\Pi_h^{\rm div}\bm{B}_0)=
 \Pi_h^{0} (\nabla\cdot\bm{B}_0)=0. 
$$

\begin{theorem} 
Assume that {$\bm{X}\times Q$} is given by
  \eqref{space-h}.  At each time step {$n$}, any solution {$(\bm{u}^n,
  \bm{E}^{n}, \bm{B}^n)\in \bm{X}$} of the fully-discrete finite
  element scheme Problem~\ref{prob:nonlinear} satisfies the Gauss law
  exactly: 
$$
\nabla\cdot \bm{B}^n=0. 
$$
\end{theorem}

\begin{proof}
By \eqref{nonlinear3}, we have 
$$
\frac{\bm{B}^n-\bm{B}^{n-1}}{k}+\nabla\times\bm{E}^n=0.
$$
Taking divergence on both sides, we have 
$$
\nabla\cdot(\bm{B}^n-\bm{B}^{n-1})=0.
$$
The desired result then follows by induction and \eqref{B0hdiv0}.
\end{proof}

\subsection{Energy estimates}
Next, we establish some energy estimates for both continuous and
discrete MHD systems. An energy estimate often refers to an a
prior estimate for the solution to a system of partial differential equations for
many physical systems including MHD.  In view of physical properties,
it shows the conservation or decay of the total energy of the
physical system.  For mathematical qualitative analysis, it provides
crucial technical tools for understanding stability and well-posedness
of the underlying PDE.  For numerical analysis, it gives us
insight and guidance to the design of appropriate discretization
schemes that inherit the energy estimate from the continuous level.  More
importantly, the technical process of deriving the energy estimate 
also provides clues about the relationships between different physical quantities.

\begin{theorem} \label{thm:energy0}
For any {$(\bm{u}, \bm{B}, \bm{E})\in \bm{X}$} and {$p\in
  Q$} that satisfy \eqref{mom2}-\eqref{div_u_free2}, the following
  energy estimates hold
\begin{equation}
\label{energy-identity}
{1 \over 2}\frac{d}{dt}\|\bm{u}\|^2 + \frac{S}{2 Rm}\frac{d}{dt}\|\bm{B}\|^2
+ \frac{1}{Re} \|\nabla \bm{u}\|^2  + S \|\bm{j}\|^{2} 
= (\bm{f}, \bm{u}). 
\end{equation}
and
\begin{align*}
   & \max_{0\leq t\leq T} \left(
    \|\bm{u}\|^2 + \frac{S}{Rm} \|\bm{B}\|^2 \right) 
    + \frac{1}{Re} \int_{0}^{T} \|\nabla
    \bm{u}\|^2\mathrm{d}\tau + 2 S \int_{0}^{T}
    \|\bm{j}\|^2 \mathrm{d}\tau \\
    \leq & ~ \|\bm{u}_{0}\|_{L^{2}}^{2} + \frac{S}{Rm} \|\bm{B}_{0}\|_{L^{2}}^{2} + 
    {Re} \int_0^T\|\bm{f}\|_{-1}^{2}\mathrm{d}\tau
\end{align*}
where {$\bm{u}_{0}$} and {$\bm{B}_{0}$} are the given initial data. 
\end{theorem}

\begin{proof}
Taking {$\bm{v}=\bm{u}$} in \eqref{mom2}, we obtain
\begin{align}\label{energy1}
{1 \over 2}\frac{d}{dt}\|\bm{u}\|^2 + \frac{1}{Re} \|\nabla \bm{u}\|^2 
= - S (\bm{j}, \bm{u}\times \bm{B}) + (\bm{f}, \bm{u}).
\end{align}
Taking {$\bm{C}=\bm{B}$} in \eqref{induction2} and {$\bm{F}=\bm{E}$} in \eqref{curl_B2} , we obtain
\begin{equation}\label{energy2}
\frac{1}{2}\frac{d}{dt}\|\bm{B}\|^2 = - (\nabla\times \bm{E}, \bm{B})
=- Rm (\bm{j}, \bm{E}).
\end{equation}
Adding \eqref{energy1} and \eqref{energy2} and using Ohm's law \eqref{ohm1},
\begin{align*}
{1 \over 2}\frac{d}{dt}\|\bm{u}\|^2 + \frac{S}{2 Rm} \frac{d}{dt}\|\bm{B}\|^2
+ \frac{1}{Re} \|\nabla \bm{u}\|^2 
&= - S (\bm{j},\bm{E}+ \bm{u}\times \bm{B}) + (\bm{f}, \bm{u}) \\
&= - S \|\bm{j}\|^{2} + (\bm{f}, \bm{u}),  
\end{align*}
which leads to \eqref{energy-identity}. 

Using the simple inequality, 
$$
(\bm{f}, \bm{u})\le \|\bm{f}\|_{-1,\Omega}\|\nabla\bm{u}\|_{0,\Omega}
\le
{Re \over2}\|\bm{f}\|_{-1,\Omega}^2+{1 \over 2Re }\|\nabla\bm{u}\|_{0,\Omega}^2, 
$$
we obtain
$$
\frac{1}{2}\frac{d}{dt} \left( \|\bm{u}\|^2
+ \frac{S}{Rm} \|\bm{B}\|^2 \right) + \frac{1}{2Re}\|\nabla
\bm{u}\|^2+ S \|\bm{j}\|^2 \leq \frac{Re}{2}
\|\bm{f}\|_{-1}^{2}. 
$$
The second estimates then follow easily. 
\end{proof}

The above energy estimate for the continuous MHD system
\eqref{mom1}-\eqref{div_B_free1} is well-known, see e.g.
\cite{Liu.J;Wang.W.2001a,Liu.J;Wang.W.2004a,Liu.C.2009a}.  We have designed our finite
element discrete scheme in such a way that a similar cancellation also
occurs on the discrete level, as a result we have extended this
estimate to finite element semi-discrete system
\eqref{mom1}-\eqref{div_B_free1}.

Similarly, we have the following energy estimate for Problem \ref{prob:nonlinear}. 
\begin{theorem}\label{nonlinear_energy} 
  For any {$(\bm{u}^n, \bm{B}^n, \bm{E}^n)\in \bm{X}$} and {$p^n\in Q$}
  that satisfy \eqref{nonlinear1}-\eqref{nonlinear4} the following
  energy estimates hold
\begin{align*}
& \lefteqn{ \|\bm{u}^{n}\|^2 + \frac{S}{Rm} \|\bm{B}^{n}\|^2+ \frac{2k}{Re}
  \|\nabla \bm{u}^n\|^2  
  + 2k S \|\bm{j}^n\|^{2}}\\
\leq& ~ \|\bm{u}^{n-1}\|^2 + \frac{S}{Rm} \|\bm{B}^{n-1}\|^2 + 2 k (\bm{f}^{n}, \bm{u}^{n}). 
\end{align*}
and
\begin{align*}
& \max_{0\leq j\leq n} \left( \|\bm{u}^{j}\|^2 + \frac{S}{Rm} \|\bm{B}^{j}\|^2 \right) 
+ \sum_{j=1}^{n} \frac{k}{Re} \|\nabla \bm{u}^{j}\|^2
+ 2k S \sum_{i=1}^{n} \|\bm{j}^{i}\|^2 \\
\leq & ~ \|\bm{u}^{0}\|^2 + \frac{S}{Rm} \|\bm{B}^{0}\|^2 + k \sum_{j=1}^{n} {Re}\|\bm{f}^{j}\|_{-1}^{2}.
\end{align*}

\end{theorem}
The proof of the above theorem is omitted here as it is analogous to
that of Theorem \ref{thm:energy0} and of Theorem \ref{thm:energy1}
below.
%
%

\subsection{Linearization of the nonlinear Problem
  \ref{prob:nonlinear}}
For each time-step, Problem \ref{prob:nonlinear} is a system of
nonlinear equations.  We can use either Picard or Newton iteration
or a combination of the two to linearize these nonlinear systems.

Picard linearization is obtained by fixing certain variables of the
nonlinear terms and solve the remaining linear terms; this does not 
yield a unique linearization. We prove that the following Picard 
linearization scheme has some desirable mathematical properties. 
\begin{algorithm}[Picard iteration]\label{alg:picard}
Given
$$
(\bm{u}^{n,0},\bm{E}^{n,0},\bm{B}^{n,0},p^{n,0})
=(\bm{u}^{n-1},\bm{E}^{n-1},\bm{B}^{n-1},p^{n-1}),
$$
find
$
(\bm{u}^{n,m},\bm{E}^{n,m},\bm{B}^{n,m},p^{n,m})\in \bm{X}\times Q
$ (for {$m=1,2,3,\ldots$} ), such that for any 
{$(\bm{v},\bm{F},\bm{C},q)\in \bm{X}\times Q$},
\begin{dgroup*}
\begin{dmath*}
\left( \frac{\bm{u}^{n,m}-\bm{u}^{n-1}}{k},\bm{v} \right) 
+  {1 \over 2} [(\bm{u}^{n,m-1}\cdot \nabla\bm{u}^{n,m},\bm{v})
- (\bm{u}^{n,m-1}\cdot \nabla\bm{v}, \bm{u}^{n,m})]
+ \frac{1}{Re} (\nabla \bm{u}^{n,m}, \nabla \bm{v})  
- S (\bm{j}^{n,m}_{n,m-1}\times \bm{B}^{n,m-1},\bm{v})
- (p^{n,m},\nabla\cdot \bm{v}) = (\bm{f}^{n},\bm{v}),   
\end{dmath*}
\begin{dmath*}
(\bm{j}^{n,m}_{n,m-1},\bm{F}) 
- \frac{1}{Rm} (\bm{B}^{n,m}, \nabla\times \bm{F}) = 0, 
\end{dmath*}
\begin{dmath*}
\left( \frac{\bm{B}^{n,m}-\bm{B}^{n-1}}{k}, \bm{C} \right) 
+ (\nabla \times \bm{E}^{n,m}, \bm{C} ) = 0,   
\end{dmath*}
\begin{dmath*}
(\nabla\cdot \bm{u}^{n,m}, q) =0.
\end{dmath*}
\end{dgroup*}
where {$\bm{j}^{n,m}_{n,m-1} := \bm{E}^{n,m}+\bm{u}^{n,m}\times \bm{B}^{n,m-1} $}.
\end{algorithm}

Unlike Picard linearization, the Newton linearization technique gives a unique
linearization, by simply taking the Fr\'{e}chet derivatives of
the nonlinear terms in \eqref{mom2}-\eqref{curl_B2}.  We choose not to
modify the convection term {$( \bm{u}\cdot \nabla \bm{u}, \bm{v})$} by
\eqref{modified-convection} in our Newton linearization scheme because
such a modification does not improve the formulation's mathematical
properties.

\begin{algorithm}[Newton iteration]
\label{alg:newton}
Given 
$$
(\bm{u}^{n,0},\bm{E}^{n,0},\bm{B}^{n,0},p^{n,0})
=(\bm{u}^{n-1},\bm{E}^{n-1},\bm{B}^{n-1},p^{n-1}),
$$
find {$(\bm{\xi}^{n,m},p^{n,m})\in \bm{X}\times Q$} 
(for {$m=1,2,3,\ldots$}), such that for any 
{$(\bm{\eta}, q)\in \bm{X}\times Q$},
\begin{dgroup*}
\begin{dmath*}
\left( \frac{\bm{u}^{n,m}-\bm{u}^{n-1}}{k},\bm{v} \right) 
+ (\bm{u}^{n,m-1}\cdot \nabla\bm{u}^{n,m},\bm{v})
+ (\bm{u}^{n,m}\cdot \nabla\bm{u}^{n,m-1}, \bm{v})
+ \frac{1}{Re} (\nabla \bm{u}^{n,m}, \nabla \bm{v})
- S ((\bm{E}^{n,m}+\bm{u}^{n,m}\times \bm{B}^{n,m-1}
+ \bm{u}^{n,m-1}\times \bm{B}^{n,m})\times \bm{B}^{n,m-1},\bm{v})
- S ((\bm{E}^{n,m-1}+\bm{u}^{n,m-1}\times \bm{B}^{n,m-1})\times \bm{B}^{n,m}, \bm{v}) 
- (p^{n,m},\nabla\cdot \bm{v}) 
= (\bm{f}^{n,m}_{N},\bm{v}),
\end{dmath*}
\begin{dmath*}
(\bm{E}^{n,m}+\bm{u}^{n,m}\times \bm{B}^{n,m-1}
+ \bm{u}^{n,m-1}\times \bm{B}^{n,m} ,\bm{F})
- \frac{1}{Rm} ( \bm{B}^{n,m}, \nabla\times \bm{F}) 
=(\bm{\phi}_{N}^{n,m}, \bm{F}), 
\end{dmath*}
\begin{dmath*}
\left( \frac{ \bm{B}^{n,m}-\bm{B}^{n-1}}{k}, \bm{C} \right) 
+ (\nabla\times \bm{E}^{n,m}, \bm{C}) = 0, 
\end{dmath*}
\begin{dmath*}
(\nabla\cdot \bm{u}^{n,m}, q) =0.
\end{dmath*}
\end{dgroup*}
where 

\begin{align*}
\bm{f}_{N}^{n,m}
&= \bm{f}^{n} + (\bm{u}^{n,m-1}\cdot
\nabla)\bm{u}^{n,m-1} - S \bm{E} ^{n,m-1}\times {\bm{B}} ^{n,m-1}\\
&- 2 S ({\bm{u}} ^{n,m-1}\times {\bm{B}}^{n,m-1})\times {\bm{B}}
^{n,m-1},\\
\bm{\phi}_{N}^{n,m} &= {\bm{u}} ^{n,m-1}\times{\bm{B}}^{n,m-1}.
\end{align*}

\end{algorithm}

\begin{theorem}
  Let {$\bm{X}\times Q$} be given as in \eqref{space-h}.  Then, for
  sufficiently small {$k$}, both Algorithm \ref{alg:picard} and
  Algorithm \ref{alg:newton} are well-defined and have unique sequence
  of solutions.
$$
(\bm{u}^{n,m}, \bm{E}^{n,m}, \bm{B}^{n,m}, p^{n,m}), \quad 0\le n\le T/k
$$
satisfying
  \begin{equation}
    \label{divB0}
\nabla\cdot \bm{B}^{n,m}=0, \quad \forall n\ge 0, m\ge 0. 
  \end{equation}
\end{theorem}
The above theorem is similar to Theorem \ref{thm:picard} and has an 
analogous proof. 

We would like to point out that, using a fixed-point argument similar
to that in \cite{Lee.Y;Xu.J;Zhang.C.2011a}, it is possible to establish the following
results: At each time step, if {$k$} is sufficiently small,
  \begin{enumerate}
  \item the nonlinear problem on discrete level
    \eqref{nonlinear1}-\eqref{nonlinear4} has a solution in
    {$\bm{X}_{h}\times {Q}_{h}$}.
  \item Both Algorithm \ref{alg:picard} and Algorithm \ref{alg:newton}
    converge to a solution of Problem \ref{prob:nonlinear} as
    {$m\rightarrow \infty$}.
  \end{enumerate}
The proof for these results are quite technical and lengthy and are not 
included in this paper. In the rest of this paper, we 
focus on the linearized discretization schemes based on Picard
and Newton linearization. 

\section{Linearized discrete schemes}\label{sec:linearization}
In the previous section, we proposed Picard and Newton methods as
iterative linearization schemes to solve the nonlinear discrete Problem
\ref{prob:nonlinear}. Instead of solving the discrete nonlinear Problem
\ref{prob:nonlinear}, however, we can also use the Picard and Newton methods as
a single step linearized discretization scheme at each time step. 

\subsection{Finite element discretization based on Picard and Newton linearization}
Similar to Algorithm \ref{alg:picard}, we propose the following
linearized discrete scheme by the Picard method. 
\begin{problem}\label{prob:picard}
Given $(\bm{u}^{n-1},\bm{E}^{n-1},\bm{B}^{n-1},p^{n-1})\in \bm{X}\times Q$, 
find {$(\bm{u}^{n},\bm{E}^n,\bm{B}^n,p^{n})\in \bm{X}\times Q$} such that 
for any {$(\bm{v},\bm{F},\bm{C},q)\in \bm{X}\times Q$},
\begin{dgroup*}
\begin{dmath}
\left( \frac{\bm{u}^{n}-\bm{u}^{n-1}}{k},\bm{v} \right) 
+  {1 \over 2} [(\bm{u}^{n-1}\cdot \nabla\bm{u}^{n},\bm{v})-(\bm{u}^{n-1}\cdot \nabla\bm{v}, \bm{u}^{n})]
+ \frac{1}{Re} (\nabla \bm{u}^{n}, \nabla \bm{v})
- S (\bm{j}^{n}_{n-1}\times \bm{B}^{n-1},\bm{v} )
- (p^{n},\nabla\cdot \bm{v}) 
= (\bm{f}^{n},\bm{v}),  \label{picard01} 
\end{dmath}
\begin{dmath}
(\bm{j}^{n}_{n-1},\bm{F}) 
- \frac{1}{Rm} ( \bm{B}^{n}, \nabla\times \bm{F}) =0, \label{picard02} 
\end{dmath}
\begin{dmath}
\left( \frac{\bm{B}^{n}-\bm{B}^{n-1}}{k}, \bm{C} \right) 
+ (\nabla\times \bm{E}^{n}, \bm{C} ) = 0, \label{picard03}
\end{dmath}
\begin{dmath}
(\nabla\cdot \bm{u}^{n}, q)  = 0, \label{picard04} 
\end{dmath}
\end{dgroup*}
where {$\bm{j}^{n}_{n-1} := \bm{E}^{n}+\bm{u}^{n}\times \bm{B}^{n-1} $}.
\end{problem}

By making the convection term explicit, a simplified
Picard linearization can be obtained,
\begin{problem}\label{prob:picard_explicit}
Given $(\bm{u}^{n-1},\bm{E}^{n-1},\bm{B}^{n-1},p^{n-1})\in \bm{X}\times Q$, 
find {$(\bm{u}^{n},\bm{E}^n,\bm{B}^n,p^{n})\in \bm{X}\times Q$} such that 
for any {$(\bm{v},\bm{F},\bm{C},q)\in \bm{X}\times Q$},
\begin{dgroup*}
\begin{dmath*}
\left( \frac{\bm{u}^{n}-\bm{u}^{n-1}}{k},\bm{v} \right) 
+  {1 \over 2} [(\bm{u}^{n-1}\cdot \nabla\bm{u}^{n-1},\bm{v})-(\bm{u}^{n-1}\cdot \nabla\bm{v}, \bm{u}^{n-1})]
+ \frac{1}{Re} (\nabla \bm{u}^{n}, \nabla \bm{v})
- S (\bm{j}^{n}_{n-1}\times \bm{B}^{n-1},\bm{v} )
- (p^{n},\nabla\cdot \bm{v}) = (\bm{f}^{n},\bm{v}).
\end{dmath*}
\begin{dmath*}
(\bm{j}^{n}_{n-1},\bm{F}) 
- \frac{1}{Rm} ( \bm{B}^{n}, \nabla\times \bm{F}) =0,
\end{dmath*}
\begin{dmath*}
\left( \frac{\bm{B}^{n}-\bm{B}^{n-1}}{k}, \bm{C} \right) 
+ (\nabla\times \bm{E}^{n}, \bm{C} ) = 0, 
\end{dmath*}
\begin{dmath*}
(\nabla\cdot \bm{u}^{n}, q)  = 0,  
\end{dmath*}
\end{dgroup*}
with the same $\bm{j}_{n-1}^{n}$ as that in Problem \ref{prob:picard}.
\end{problem}
One feature of this scheme is that the underlying stiffness matrix 
is symmetrized. 

Analogous to Algorithm \ref{alg:newton}, we propose the following
linearized discrete scheme by Newton method. 
\begin{problem}
\label{prob:newton0} 
Given $(\bm{u}^{n-1},\bm{E}^{n-1},\bm{B}^{n-1},p^{n-1})\in \bm{X}\times Q$, 
find {$(\bm{u}^{n},\bm{E}^n,\bm{B}^n,p^{n})\in \bm{X}\times Q$}
such that for any {$(\bm{\eta}, q)\in \bm{X}\times Q$},
\begin{dgroup*}
\begin{dmath}
\left( \frac{\bm{u}^{n}-\bm{u}^{n-1}}{k},\bm{v} \right) 
+ (\bm{u}^{n-1}\cdot \nabla\bm{u}^{n},\bm{v})
+ (\bm{u}^{n}\cdot \nabla\bm{u}^{n-1}, \bm{v})
+ \frac{1}{Re} (\nabla \bm{u}^{n}, \nabla \bm{v})
- S ((\bm{E}^{n}+\bm{u}^{n}\times \bm{B}^{n-1}
+ \bm{u}^{n-1}\times \bm{B}^{n})\times \bm{B}^{n-1},\bm{v})
- S ((\bm{E}^{n-1}+\bm{u}^{n-1}\times \bm{B}^{n-1})\times \bm{B}^{n}, \bm{v}) 
 - (p^{n},\nabla\cdot \bm{v}) = (\bm{f}^{n}_{N},\bm{v}),\label{newton1} 
\end{dmath}
\begin{dmath}
(\bm{E}^{n} + \bm{u}^{n}\times \bm{B}^{n-1}
+ \bm{u}^{n-1}\times \bm{B}^{n} ,\bm{F}) 
- \frac{1}{Rm} ( \bm{B}^{n}, \nabla\times \bm{F}) 
= (\bm{\phi}_{N}^{n}, \bm{F}), \label{newton2} 
\end{dmath}
\begin{dmath}
\left( \frac{ \bm{B}^{n}-\bm{B}^{n-1}}{k}, \bm{C} \right) 
+ (\nabla\times \bm{E}^{n}, \bm{C}) = 0, \label{newton3}
\end{dmath}
\begin{dmath}
(\nabla\cdot \bm{u}^{n}, q) = 0, \label{newton4} 
\end{dmath}
\end{dgroup*}
where 
\begin{dgroup*}
\begin{dmath*}
\bm{f}_{N}^{n}= \bm{f}^{n}
+ (\bm{u}^{n-1}\cdot \nabla)\bm{u}^{n-1}
- S \bm{E} ^{n-1}\times {\bm{B}} ^{n-1}  
 - 2 S ({\bm{u}} ^{n-1}\times {\bm{B}}^{n-1})\times {\bm{B}} ^{n-1},
\end{dmath*}
\begin{dmath*}
\bm{\phi}_{N}^{n} = {\bm{u}} ^{n-1}\times{\bm{B}} ^{n-1}.
\end{dmath*}
\end{dgroup*} 
\end{problem}

While the above linearized schemes are valid for both the continuous and
discrete cases, we state the following theorem for the discrete scheme
only. 
\begin{theorem}\label{thm:picard} 
Given {$\bm{X}\times Q$} as in \eqref{space-h}, 
assume that 
$$
\bm{u}^0, \bm{B}^0\in L^2(\Omega),  ~k\sum_{j=1}^{n}\|\bm{f}^{j}\|_{-1}^{2}<\infty.
$$ 
Then, for  sufficiently small {$k$}, 
we have for all $n \geq 1$,
\begin{enumerate}
\item Both Problems \ref{prob:picard} and \ref{prob:newton0}
  have unique global solutions:
$$
(\bm{u}^n, \bm{E}^n, \bm{B}^n, p^n), \quad 0\le n\le T/k.
$$
\item Solutions to Problems \ref{prob:picard} and 
\ref{prob:newton0} satisfy the following property strongly:
\begin{equation}
    \label{divB0}
\nabla\cdot \bm{B}^n=0.
\end{equation}
\end{enumerate}
\end{theorem}

\begin{remark}
The quantity {$\left( k\sum_{j=1}^{n}\|\bm{f}^{j}\|_{-1}^{2} \right)^{1/2}$} is a discretization of 
{$L^{2}([0,T]; H^{-1}(\Omega))$} norm. 
\end{remark}

\begin{remark}
For the Picard iteration Problem \ref{prob:picard}, it can be proved
that the above theorem holds if 
\begin{equation}\label{small-k}  
k\lesssim h^3 \left(\|\bm{u}^0\|+   \|\bm{B}^0\|+ \left( \int_0^T\|\bm{f}\|_{-1}^2 \right)^{1/2} \right).
\end{equation}
This condition is rather stringent and whether this
constraint can be relaxed is a subject of further investigation.  We
note, however, that the discrete problem has a global solution for all
{$n$} under this condition. 
\end{remark}

As Theorem \ref{thm:picard} is a special case of Theorem
\ref{thm:picard1} and Theorem \ref{thm:newton1}, its proof is omitted.
\bigskip

In the following subsections, we first focus on Picard
linearization, its energy estimate, and well-posedness. Afterwards, 
we turn to Newton linearization, and give similar results.

\subsection{Energy estimates}
One desirable feature of our Picard linearization scheme Problem
\ref{prob:picard} is that it satisfies an energy estimate
on both continuous and discrete levels that is analogous to the
original problem as shown in Theorem \ref{thm:energy0}. 

\begin{theorem}\label{thm:energy1}
Any solution of Problem \ref{prob:picard} satisfies the following estimates
\begin{align}\nonumber
& \max_{0\leq j\leq n} \left( \|\bm{u}^{j}\|^2 + \frac{S}{Rm} \|\bm{B}^{j}\|^2 \right)
+\sum_{j=1}^{n} \frac{k}{Re} \|\nabla
\bm{u}^{j}\|^2 + 2k S \sum_{i=1}^{n} \|\bm{j}_{i-1}^{i}\|^2 \\ 
\label{Picard-energy}
\leq & ~ \|\bm{u}^{0}\|^2 + \frac{S}{Rm} \|\bm{B}^{0}\|^2 + k \sum_{j=1}^{n} Re \|\bm{f}^{j}\|_{-1}^{2}.
\end{align}
\end{theorem}

\begin{proof}
Take {$\bm{v}=\bm{u}^{n}$} in \eqref{picard01},
\begin{align}\label{energypicard1}
\left( \frac{ \bm{u}^{n}-\bm{u}^{n-1}}{k}, \bm{u}^{n} \right)
  + \frac{1}{Re} (\nabla \bm{u}^{n},\nabla \bm{u}^{n}) 
  = - S (\bm{j}_{n-1}^{n},
  \bm{u}^{n}\times \bm{B}^{n-1}) + (\bm{f}^{n}, \bm{u}^{n}).
\end{align}
Take {$\bm{C}=\bm{B}^{n}$} in \eqref{picard03} and
{$\bm{F}=\bm{E}^{n}$} in \eqref{picard02},
\begin{gather}\label{energypicard2}
\left( \frac{ \bm{B}^{n}-\bm{B}^{n-1} }{k}, \bm{B}^{n} \right)
  = - ( \nabla\times \bm{E}^{n}, \bm{B}^{n})
  = - Rm (\bm{j}_{n-1}^{n}, \bm{E}^{n}).
\end{gather}
Adding \eqref{energypicard1} and \eqref{energypicard2} to eliminate
the Lorentz force, and using Ohm's law:
\begin{dmath*}
  \left( \frac{ \bm{u}^{n}-\bm{u}^{n-1} }{k}, \bm{u}^{n} \right)
  + \frac{S}{Rm} \left( \frac{ \bm{B}^{n}-\bm{B}^{n-1} }{k}, \bm{B}^{n} \right)
  + \frac{1}{Re} (\nabla \bm{u}^{n}, \nabla \bm{u}^{n})
  + S (\bm{j}_{n-1}^{n},\bm{j}_{n-1}^{n}) 
  =( \bm{f}^{n}, \bm{u}^{n}).
\end{dmath*}
The desired results follow by combining the above two estimates with
the following simple inequality:
$$  
\left( \frac{\bm{u}^{n}-\bm{u}^{n-1}}{k}, \bm{u}^{n} \right)
\geq\frac{1}{2k}(\|\bm{u}^{n}\|^{2} - \|\bm{u}^{n-1}\|^{2} ).
$$
\end{proof}
For Newton linearization, energy estimates are not as neat as those
for Picard linearization.  Under some appropriate assumptions, for
sufficiently small {$k$}, we can establish the following energy estimates for
any solution of \eqref{newton1}-\eqref{newton4}:
$$
\mathcal{E}^{n} \leq e^{M_{1}T}(Ck(1+M_{2}k)^{-1}\sum_{i=1}^{n}\|\bm{f}^{i}\|_{-1}^{2}+\mathcal{E}^{0})
$$
where energy {$\mathcal{E}^{n}$} is defined as
$$
\mathcal{E}^{n}:=\frac{1}{2} \left(
\|\bm{u}^{n}\|^{2} + \frac{S}{Rm} \|\bm{B}^{n}\|^{2} \right) 
+ k(1+M_{2}k)^{-1} \|\bm{E}^{n}\|^{2} + \frac{k(1+M_{2}k)^{-1} }{2 Re} \|\nabla\bm{u}^{n}\|^{2}
$$
Here {$M_{1}$} and {$M_{2}$} are positive constants which only depend on
{$\bm{B}^{n-1}$}, {$\bm{u}^{n-1}$}, {$\bm{E}^{n-1}$} and {$k$}.

\subsection{Mixed formulations}
In this subsection, we formulate Algorithm \ref{alg:picard} for
Problem \ref{prob:picard} as a mixed problem and then
establish its well-posedness.  We use {$\bm{u}^{-}=\bm{u}^{n-1}$}
and {$\bm{B}^{-}=\bm{B}^{n-1}$} to denote known velocity and magnetic
field, either from the previous time step or iteration step.

\subsubsection{Two mixed formulations for Picard and Newton linearizations}
Given {$(\bm{u}^-,\bm{E}^-,\bm{B}^-)\in\bm{X}$}, we define
$$
\bm{d}(\bm{u}, \bm{v}):= {1 \over 2} [(\bm{u}^{-}\cdot
\nabla\bm{u},\bm{v})-(\bm{u}^{-}\cdot \nabla\bm{v}, \bm{u})] + \frac{1}{Re}
(\nabla \bm{u}, \nabla \bm{v}).
$$
For {$\bm{\xi}=(\bm{u,E,B})$}, {$\bm{\eta}=(\bm{v,F,C})\in \bm{X}$} and {$p,
q\in Q$}, define bilinear forms $\bm{a}_{0}(\cdot, \cdot)$, 
{$\bm{a}(\cdot, \cdot)$} on {$\bm{X}\times\bm{X}$}, and {$\bm{b}(\cdot, \cdot)$}
on {$\bm{X}\times Q$}, by
\begin{align*}
\bm{a}_0(\bm{\xi}, \bm{\eta})
& := k^{-1} (\bm{u}, \bm{v}) + \bm{d}(\bm{u}, \bm{v}) 
+ S (\bm{u}\times {\bm{B}} ^{-}, \bm{v}\times {\bm{B}} ^{-})\\
&
+ S (\bm{v}\times {\bm{B}} ^{-}, \bm{E})
+ S (\bm{u}\times {\bm{B}} ^{-}, \bm{F})
+ S (\bm{E}, \bm{F}) \\
& 
- \frac{S}{Rm} (\bm{B}, \nabla\times \bm{F})
+ \frac{S}{Rm} (\nabla \times \bm{E}, \bm{C}) 
+ \frac{S k^{-1}}{Rm} (\bm{B}, \bm{C}) 
\end{align*}
and 
$$
\bm{a}(\bm{\xi}, \bm{\eta}) :=\bm{a}_0(\bm{\xi}, \bm{\eta})
+ \frac{S}{Rm} (\nabla\cdot \bm{B}, \nabla\cdot \bm{C}),  
\quad \mbox{and} \quad 
\bm{b}(\bm{\eta}; q):= (\nabla\cdot \bm{v}, q). 
$$

For Newton linearization, we define bilinear form {$\bm{a}_{N,0}$} as:
\begin{align*}
\bm{a}_{N,0}(\bm{\xi}, \bm{\eta})
& :=  {k^{-1}} ( \bm{u}, \bm{v})
+ \bm{d}_N(\bm{u}, \bm{v}) 
- S (\bm{E} ^{-}\times \bm{B}, \bm{v}) \\
& 
- S (({\bm{u}} ^{-}\times {\bm{B}} ^{-})\times \bm{B}, \bm{v})
- S (({\bm{u}} ^{-}\times \bm{B}) \times {\bm{B}} ^{-}, \bm{v})  \\
&
+ S (\bm{E} + \bm{u}\times {\bm{B}} ^{-}, \bm{F} + \bm{v}\times {\bm{B}} ^{-})\
- \frac{S}{Rm} (\bm{B}, \nabla\times \bm{F}) \\
& 
+ \frac{S}{Rm} (\nabla\times \bm{E}, \bm{C})
+ \frac{S k^{-1}}{Rm} (\bm{B}, \bm{C})
+ S ({\bm{u}} ^{-}\times \bm{B}, \bm{F}).
\end{align*}
where
$$
\bm{d}_N(\bm{u}, \bm{v}):= (\bm{u}^{-}\cdot\nabla\bm{u},
\bm{v}) + (\bm{u}\cdot\nabla\bm{u}^{-},
\bm{v})+\frac{1}{Re} (\nabla\bm{u}, \nabla\bm{v}).
$$
And
$$
\bm{a}_{N}(\bm{\xi}, \bm{\eta}) := \bm{a}_{N,0}(\bm{\xi}, \bm{\eta}) 
+ \frac{S}{Rm} (\nabla\cdot \bm{B}, \nabla\cdot \bm{C}).
$$

\paragraph{Picard methods.}

We consider the following problem as a general model of Algorithm
\ref{alg:picard} for Problem \ref{prob:picard}.

\begin{problem}\label{prob:mixed0}  
Given {$\bm{h}=(\bm{f}, \bm{r}, \bm{l})\in [\bm{X}^{\bm{B}, 0}]^{\ast} $} satisfying\footnote{By Riesz representation theorem, $[\bm{V}^{d, 0}]^{\ast} \cong \bm{V}^{d,0}$ through the pairing \eqref{ldiv0}.}
\begin{equation}\label{ldiv0}
\langle \bm{l}, \bm{C} \rangle=(\bm{l}_{R}, \bm{C}),~\forall \bm{C}\in \bm{V}^{d}
\mbox{ ~~for some }
\bm{l}_{R}\in \bm{V}^{d,0}. 
\end{equation}
 and {$g\in Q^*$}, find {$(\bm{\xi}, p)\in
  \bm{X} \times Q$}, such that   
\begin{align}\label{brezzip01}
\bm{a}_0(\bm{\xi}, \bm{\eta})+ \bm{b}(\bm{\eta}, p)&=\langle \bm{h}, \bm{\eta} \rangle ,
\quad \forall~       
\bm{\eta} \in \bm{X}, \\\label{brezzip02}
\bm{b}(\bm{\xi},q)&=\langle {g}, q \rangle, \quad\forall q\in Q.
\end{align}
\end{problem}

We also give an equivalent problem, for which the equivalence and its well-posedness will be shown below.
\begin{problem}\label{prob:mixed}  
Given {$\bm{h}\in \bm{X}^{\ast}$} and {$g\in Q^{\ast}$}, find {$(\bm{\xi}, p)\in
\bm{X}\times Q$}, such that
\begin{align}\label{brezzip1}
\bm{a}(\bm{\xi}, \bm{\eta})+ \bm{b}(\bm{\eta}, p)&=\langle \bm{h}, \bm{\eta} \rangle ,\quad \forall~ \bm{\eta} \in \bm{X}, \\\label{brezzip2}
\bm{b}(\bm{\xi},q)&=\langle {g}, q \rangle, \quad\forall q\in Q.
\end{align}
\end{problem}

As an immediate observation of Problem \ref{prob:mixed}, we have Lemma 3.
\begin{lemma}\label{lem:graddiv0}
Assume {$\bm{h}$} satisfies \eqref{ldiv0}. The solution of Problem \ref{prob:mixed} satisfies 
$$
\nabla\cdot \bm{B}=0
$$
provided {$\nabla\cdot \bm{B}^{-}=0$}.
\end{lemma}
\begin{proof}
We also use {$\bm{V}^{d}$} to denote both continuous and discrete level.  By \eqref{brezzip1}, we get
$$
\left(\bm{l}_{R}-k^{-1}(\bm{B}-\bm{B}^{-})-\nabla\times
  \bm{E},\bm{C}\right)=(\nabla\cdot \bm{B}, \nabla\cdot\bm{C}\rangle ,\quad\forall \bm{C}\in \bm{V}^{d}.
$$
Now, take 
$
\bm{C}=\bm{l}_{R}-k^{-1}(\bm{B}-\bm{B}^{-})-\nabla\times \bm{E}
$, we have
$$
\|\bm{C}\|^2=(\nabla\cdot \bm{B}, \nabla\cdot\bm{C}\rangle =-k^{-1}\|\nabla\cdot \bm{B}\|^2. 
$$
This implies that
$$
\nabla\cdot\bm{B}=0. 
$$
\end{proof}

\begin{theorem}\label{thm:picard1}
Problem \ref{prob:mixed} is well-posed, if {$\bm{B}^{-} \in
  L^{\infty}(\Omega)$}, {$\bm{u}^{-}\in L^{3}(\Omega)$}, and {$k$} is sufficiently small:
$$
k \leq \frac{1}{8 S }\|B ^{-}\|_{0,\infty}^{-2}. 
$$
More precisely, for any {$\bm{h}\in \bm{X}^{\ast}$} and {$g\in Q^{\ast}$},  there
is a unique {$(\bm{\xi}, p)=(\bm{u},\bm{E}, \bm{B},p)\in \bm{X}\times Q$} that solves Problem
\ref{prob:mixed} and satisfies:
\begin{equation}  \label{stability}
\|(\bm{u},\bm{E},\bm{B})\|_{\bm{X}}+  \|p\|_{Q}\lesssim \|\bm{h}\|_{\bm{X}^*}+\|g\|_{Q^*}.
\end{equation}
\end{theorem}
The proof of the above theorem is given in the following
subsection. 

As a result of Lemma \ref{lem:graddiv0} and well-posedness of Problem \ref{prob:mixed}, we have theorem \ref{thm:mixed0}.
\begin{theorem}\label{thm:mixed0}
Problem \ref{prob:mixed0} is well-posed. 
\end{theorem}
 
\begin{proof}
By \eqref{ldiv0}, we have {$\bm{h}\in \bm{X}^{\ast}$}. With such data {$\bm{h}$} and {$g$}, 
by Theorem
  \ref{thm:picard1}, Problem \ref{prob:mixed} has a unique solution
  which, thanks to Lemma \ref{lem:graddiv0}, is also a solution of
  Problem \ref{prob:mixed0}.  This proves the existence of solution
  for Problem \ref{prob:mixed0}. 
  
  On the other hand, any solution of Problem \ref{prob:mixed0} must be
  a solution of Problem \ref{prob:mixed} with the same data.  The solution to Problem \ref{prob:mixed} is unique, thus the
  solution to Problem \ref{prob:mixed0} must be unique.
\end{proof} 
 
By a similar argument, we get the following result on equivalence:

\begin{lemma}\label{lem:equivalence}
{$(\bm{u}^{n},\bm{E}^n,\bm{B}^n,p^{n})\in \bm{X}\times Q$} solves
  Problem \ref{prob:picard} if and only if
$$
(\bm{\xi},p)=(\bm{u}^{n},\bm{E}^n,\bm{B}^n,p^{n})
$$
solves both Problem \ref{prob:mixed} and Problem \ref{prob:mixed0} with
$$
(\bm{u}^-,\bm{E}^-,\bm{B}^-)  =
(\bm{u}^{n-1},\bm{E}^{n-1},\bm{B}^{n-1})
$$
and 
$$
\bm{h} = \left( \bm{f} + k^{-1}\bm{u}^{-}, \bm{0}, \frac{S k^{-1}}{Rm} \bm{B}^{-} \right), ~~ g=0. 
$$
\end{lemma}

\bigskip

\paragraph{Newton methods.}

Similar to Picard linearization, we reformulate the Newton iteration scheme
into a mixed formulation. 
\begin{problem}\label{prob:newton}
Given {$\bm{h}\in {\bm{X}}^{\ast}$} satisfying \eqref{ldiv0} and {$g\in Q^{\ast}$},  find {$(\bm{\xi}, p)\in \bm{X}\times Q$}, such that for any {$(\bm{\eta}, q)\in \bm{X}\times Q$},
\begin{align}\label{picardvar1}
\bm{a}_{N}(\bm{\xi}, \bm{\eta})+ \bm{b}_{N}(\bm{\eta} , p)&=\langle \bm{h}, \bm{\eta} \rangle , \\\label{picardvar2}
\bm{b}(\bm{\xi}, q)&=\langle g, q \rangle.
\end{align}
\end{problem}

\begin{problem}\label{prob:newton00}
Given {$\bm{h}\in {[\bm{X}^{\bm{B},0}]}^{\ast}$} satisfying \eqref{ldiv0} and {$g\in Q^{\ast}$},  find {$(\bm{\xi}, p)\in \bm{X}\times Q$}, such that for any 
{$(\bm{\eta}, q)\in \bm{X} \times Q$},
\begin{align}\label{picardvar1}
\bm{a}_{N,0}(\bm{\xi}, \bm{\eta})+ \bm{b}_{N}(\bm{\eta} , p)&=\langle \bm{h}, \bm{\eta} \rangle , \\\label{picardvar2}
\bm{b}(\bm{\xi}, q)&=\langle g, q \rangle.
\end{align}
\end{problem}

Note that the argument in Lemma \ref{lem:equivalence} only involves linear equations, hence the  equivalence can be also established for Problem \ref{prob:newton} and Problem \ref{prob:newton00}:
\begin{lemma}
{$(\bm{u}^{n},\bm{E}^n,\bm{B}^n,p^{n})\in \bm{X}\times Q$} solves
  Problem \ref{prob:newton0} if and only if  
$$
(\bm{\xi},p)=(\bm{u}^{n},\bm{E}^n,\bm{B}^n,p^{n})
$$ 
solves Problem \ref{prob:newton} and Problem \ref{prob:newton00} with
$$
(\bm{u}^-,\bm{E}^-,\bm{B}^-)  =
(\bm{u}^{n-1},\bm{E}^{n-1},\bm{B}^{n-1}),
$$
$$
\bm{h} = (\bm{f} + k^{-1}\bm{u}^{-} + ( {\bm{u}}^{-}\cdot \nabla){\bm{u}} ^{-} - S \bm{E} ^{-}\times {\bm{B}} ^{-}
- 2 S ({\bm{u}} ^{-}\times {\bm{B}} ^{-})\times {\bm{B}} ^{-},
  S \bm{u}^{-}\times\bm{B}^{-}, \frac{S k^{-1}}{Rm} \bm{B}^{-} ),
$$
and
$$
	g=0.
$$
\end{lemma}

In the next subsection, we prove the well-posedness of Problem \ref{prob:newton}.
\begin{theorem}\label{thm:newton1}
 Assume the known {${\bm{u}} ^{-}$}, {$\nabla \bm{u}^{-}$}, {$\bm{E} ^{-}$}
and {${\bm{B}} ^{-}$} belong to {$L^{\infty}$}. The Problem \ref{prob:newton} is well-posed for sufficiently small {$k$}.
\end{theorem}

\subsection{Proof of the well-posedness} 
This subsection is devoted to the proof of Theorem \ref{thm:picard1} and Theorem \ref{thm:newton1}. By Brezzi's theory \cite{Brezzi.F.1974a,Boffi.D;Brezzi.F;Fortin.M.2013a}, the proofs of these theorems are reduced to proving the following statements:
\begin{enumerate}
\item
{$\bm{a}(\cdot,\cdot)$} and {$ \bm{b}(\cdot, \cdot)$} are bounded;
\item
an inf-sup condition holds for {$\bm{b}(\cdot, \cdot)$};
\item
an inf-sup condition holds for {$\bm{a}(\cdot,\cdot)$} in the kernel of the operator induced by {$\bm{b}$}.
\end{enumerate}

\subsubsection{Picard methods}
Let us first analyze the Picard methods.

\begin{lemma}\label{boundedness}
Assume {$\bm{B} ^{-}\in L^3(\Omega)$}, {$\bm{u}^{-}\in L^{3}(\Omega)$}.
Then {$\bm{a}(\cdot, \cdot)$} and {$\bm{b}(\cdot, \cdot)$} are bounded linear operators:
$$
\bm{a}(\bm{\xi}, \bm{\eta}) \le C\|\bm{\xi}\|_{\bm{X}}\|\bm{\eta}\|_{\bm{X}},
\quad
\bm{b}(\bm{\eta}, q)\le C \|\bm{\eta}\|_{\bm{X}}\|q\|_{Q},
$$
where the constant {$C$} depends on {$\Omega$}, {$\|\bm{B} ^{-}\|_{0,3}$}, {$\|\bm{u}^{-}\|_{0,3}$}, but not on {$k$}.
\end{lemma}

\begin{proof}
Without loss of generality, we assume that {$0 < k \leq 1$}. By the Sobolev embedding theorem,
$$
|( \bm{u}^{-}\cdot \nabla \bm{u},\bm{v} )|
\lesssim \|\bm{u}^{-} \|_{0,3}\|\bm{u}\|_{1}\|\bm{v}\|_{0,6}
\lesssim \|\bm{u}^{-}\|_{0,3}\|\bm{u}\|_{ 1,k}\|\bm{v}\|_{ 1,k}.
$$
The estimate of the term {$( (\bm{u}^{-}\cdot \nabla)\bm{v},\bm{u} )$} is similar. We also note
\begin{align*}
(\bm{u}\times {\bm{B}}^{-}, \bm{v}\times \bm{B}^{-})
\lesssim \|{\bm{B}}^{-}\|_{0,3}^{2}\|\bm{u}\|_{0,6}\|\bm{v}\|_{0,6}
\lesssim \|{\bm{B}}^{-}\|_{0,3}^{2}\|\bm{u}\|_{1,k}\|\bm{v}\|_{1,k}.
\end{align*}
Thus,
$$
d(\bm{u},\bm{v}) 
+ S (\bm{u}\times {\bm{B}}^{-}, \bm{v}\times {\bm{B}}
^{-}) \lesssim \|\bm{u}\|_{ 1,k}\|\bm{v}\|_{ 1,k}.
$$
In addition, we have the following estimates:
\begin{align*}
& (\nabla \times \bm{E}, \bm{B})
= \left( \sqrt{k}\nabla\times \bm{E}, \frac{1}{\sqrt{k}} \bm{B} \right)
\lesssim
\|\bm{E}\|_{ \mathrm{curl}, k}\|\bm{B}\|_{ \mathrm{div}, k}, \\
& (\bm{u}\times {\bm{B}} ^{-}, \bm{F})
\lesssim \|{\bm{B}} ^{-}\|_{0,3}\|\bm{u}\|_{0,6}\|\bm{F}\| 
\lesssim \|{\bm{B}} ^{-}\|_{0,3}\|\bm{u}\|_{ 1,k}\|\bm{F}\|_{ \mathrm{curl}, k}, \\
& ( \bm{E}, \bm{F})\le \|\bm{E}\|_{ \mathrm{curl}, k}\|\bm{F}\|_{ \mathrm{curl}, k}, \\
& \left( \frac{1}{\sqrt{k}}\bm{u}, \frac{1}{\sqrt{k}}\bm{v} \right) \le \|\bm{u}\|_{ 1,k}\|\bm{v}\|_{ 1,k}, \\
& \left( \frac{1}{\sqrt{k}}\bm{B}, \frac{1}{\sqrt{k}}\bm{C} \right) \le \|\bm{B}\|_{ \mathrm{div}, k}\|\bm{C}\|_{ \mathrm{div}, k}, \\
& (\nabla\cdot \bm{v}, q) = \left( \frac{1}{\sqrt{k}}\nabla\cdot \bm{v}, \sqrt{k} q \right) \lesssim  \|\bm{v}\|_{ 1,k}\|q\|_{0, k}.
\end{align*}
Therefore, the conclusion holds.
\end{proof}

We now proceed to proving the inf-sup condition of {$\bm{b}(\cdot, \cdot)$}.
\begin{lemma}\label{infsupb}
{$\bm{b}(\cdot,\cdot)$} satisfies inf-sup condition, that is, there exists constant {$\alpha>0$}, such that
$$
\inf_{q\in Q}  \sup_{\bm{\eta}\in \bm{X}}\frac{\bm{b}(\bm{\eta}, q)}{\|\bm{\eta}\|_{\bm{X}} \|q\|_{Q}  }
\geq \alpha>0.
$$
\end{lemma}

\begin{proof}
  The inf-sup condition of velocity and pressure of classical Sobolev
  spaces and stable finite element pairs is well-known 
  \cite{Boffi.D;Brezzi.F;Fortin.M.2013a}: there exists
  {$\gamma_{0}>0$}, such that 
$$
\inf_{0\neq q\in Q}\sup_{0\neq \bm{v}\in \bm{V}}\frac{(\nabla\cdot
  \bm{v},q)}{\|\bm{v}\|_{1}\|q\|_{0}}\geq \gamma_{0}>0.
$$
By definition of new norms,
$$
\|\bm{v}\|_{ 1, k}\|q\|_{0, k} \lesssim \|\bm{v}\|_{1}\|q\|_{0},
$$
since {$\nabla \cdot \bm{v}$} is a part of {$\nabla \bm{v}$}.

This implies that inf-sup condition of velocity-pressure holds: there
exists {$\alpha>0$}, such that
$$
\inf_{0\neq q\in Q}\sup_{0\neq \bm{v}\in \bm{V}}\frac{(\nabla\cdot
  \bm{v},q)}{\|\bm{v}\|_{ 1, k}\|q\|_{0, k}} \geq \alpha>0.
$$
\end{proof}

Next, we establish an inf-sup condition of {$\bm{a}(\cdot,\cdot)$}
in {$\bm{X}^{u,0}$}.

\begin{lemma}\label{picard_lemma}
Assume {${\bm{B}} ^{-}\in L^{\infty}$}, and {$k \leq \frac{1}{8 S }\|\bm{B} ^{-}\|_{0,\infty}^{-2}$}. The inf-sup conditions hold:
$$
\inf_{\bm{0} \neq \bm{\xi}\in \bm{X}^{\bm{u}, 0}} \sup_{\bm{0} \neq \bm{\eta}\in \bm{X}^{\bm{u}, 0 }}\frac{\bm{a}(\bm{\xi}, \bm{\eta}) }{\|\bm{\xi}\|_{\bm{X}} \|\bm{\eta}\|_{\bm{X}}}\geq \alpha>0,
$$
$$
\inf_{\bm{0} \neq \bm{\eta}\in \bm{X}^{\bm{u}, 0 }} \sup_{ \bm{0} \neq \bm{\xi} \in \bm{X}^{\bm{u}, 0 }}\frac{\bm{a}(\bm{\xi}, \bm{\eta}) }{\|\bm{\xi}\|_{\bm{X}} \|\bm{\eta}\|_{\bm{X}}}\geq \alpha>0.
$$
\end{lemma}

\begin{proof}
Take {$\bm{v}=\bm{u}$}, {$\bm{F}=\bm{E}$}, {$\bm{C}=\frac{1}{2}(\bm{B}+k \nabla\times \bm{E})$}:
\begin{align*}
 \bm{a}(\bm{\xi}, \bm{\eta}) 
 = & {k^{-1}}( \bm{u}, \bm{u})
 + \bm{d}(\bm{u}, \bm{u})
+ S \|\bm{E}+\bm{u}\times {\bm{B}} ^{-}\|^{2}
- \frac{S}{Rm} (\bm{B}, \nabla\times \bm{E}) 
+ \frac{Sk}{2 Rm} \|\nabla\times \bm{E}\|^{2} \\
&  + \frac{S}{2 Rm} (\nabla\times \bm{E}, \bm{B})
+ \frac{S}{2 Rm} (\nabla\times \bm{E}, \bm{B})
+ \frac{S k^{-1} }{2 Rm} (\bm{B}, \bm{B})
+ \frac{S}{2 Rm} \|\nabla\cdot \bm{B}\|^{2}\\
= & {k^{-1}} \| \bm{u}\|^{2}
+ \bm{d}(\bm{u}, \bm{u})
+ S \|\bm{E} + \bm{u}\times {\bm{B}}^{-}\|^{2}
+ \frac{Sk}{2Rm} \|\nabla\times \bm{E}\|^{2} \\
& + \frac{S k^{-1} }{2 Rm} \|\bm{B}\|^{2}
+ \frac{S}{2Rm} \|\nabla\cdot \bm{B}\|^{2}.
\end{align*}
Note the fact that
$$
\|\bm{E}+\bm{u}\times {\bm{B}} ^{-}\|^{2}\geq \frac{1}{2}\|\bm{E}\|^{2}-\|\bm{u}\times {\bm{B}} ^{-}\|^{2}.
$$
For {$k \leq \frac{1}{8 S }\|\bm{B} ^{-}\|_{\infty}^{-2}$},
$$
\frac{1}{2} k^{-1} \|\bm{u}\|^{2}\geq 4 S \|\bm{B}^{-}\|_{\infty}^{2}\|\bm{u}\|^{2}
\geq S \|\bm{u}\times \bm{B}^{-}\|^{2},
$$  
and we have
$$
{k^{-1}} \| \bm{u}\|^{2} + S \|\bm{E} + \bm{u}\times {\bm{B}} ^{-}\|^{2} 
\geq \min \left \{{1\over 2 }, {1\over 2} S \right \}({k^{-1}}\|\bm{u}\|^{2}+\|\bm{E}\|^{2}).
$$

On the other hand, there exists positive $\beta$ such that
$$
\bm{d}(\bm{u}, \bm{u})\geq \beta \lvert \bm{u} \rvert_{1}^{2},
\quad
(\nabla\cdot \bm{u}, q)=0 \quad \forall~ q,
$$
for {$\bm{\xi}\in \bm{X}^{u,0} $}.

This implies for small {$k$}, there exist constants {$\alpha$}, {$C>0$}, such that for any {$\bm{\xi}\in \bm{X}^{0, u}$}, there exists an {$\bm{\eta}\in \bm{X}^{0, u}$} satisfying
$$
\bm{a}(\bm{\xi}, \bm{\eta}) \geq \alpha \|\bm{\xi}\|_{\bm{X}}^{2},
\quad
\|\bm{\eta}\|_{\bm{X}}\le C \|\bm{\xi}\|_{\bm{X}},
$$
where {$\alpha$} and {$C$} depend on the domain {$\Omega$}, {${\bm{B}} ^{-}$} and {${\bm{u}} ^{-}$}, but not on time step size {$k$}.

The other inequality can be proved in the same way.
\end{proof}

\bigskip

Combining Lemmas \ref{boundedness}, \ref{infsupb} and
\ref{picard_lemma}, we complete the proof of Theorem \ref{thm:picard1}.

\begin{remark}
As a remark, we have assumed that
{$\bm{B}^{-}\in L^{\infty}$} in the analysis. Such an
assumption is reasonable on the discrete level. By an inverse estimate, one gets
$$
\|{\bm{B}}^{-}\|_{0,\infty}\lesssim
h^{-\frac{3}{2}}\|{\bm{B}}^{-}\|, ~
\| {\bm{u}}^{-}  \|_{0, 3} \lesssim
h^{-\frac{1}{2}} \| {\bm{u}^{-}} \|.
$$
According to the energy estimate \eqref{Picard-energy} , we know that
the {$\Vert \bm{B}^{-} \Vert$} and $\| {\bm{u}^{-}}\|$ are both uniformly bounded.
\end{remark}

\bigskip
\subsubsection{Newton methods}
Next, we prove Theorem \ref{thm:newton1} for the Newton method. 
The proof is also based on Brezzi's theory. We begin with the proof 
of the boundedness of {$\bm{a}_{N}(\cdot,\cdot)$}, which is 
quite similar to that of Picard linearization. 
\begin{lemma}
Assume {$k (\bm{E}^{-} + \bm{u}^{-}\times \bm{B}^{-})\in L^{3}(\Omega)$}, {$k^{\frac{1}{2}}\bm{u}^{-}\in L^{\infty}(\Omega)$}, {$\bm{B}^{-}\in L^{3}(\Omega)$}, {$k \nabla\bm{u}^{-}\in L^{2}(\Omega)$}. It follows that {$\bm{a}_{N}(\cdot,\cdot)$} and {$\bm{b}(\cdot, \cdot)$} are bounded in {$\bm{X}\times \bm{X}$} and {$\bm{X}\times Q$} with weighted norms.
\end{lemma}

%

The inf-sup condition of {$\bm{b}(\cdot, \cdot)$} has been shown in Lemma \ref{infsupb}. 
Now, we turn to proving the inf-sup condition of {$\bm{a}_{N}(\cdot, \cdot)$}. 
\begin{lemma}
Assume the known {${\bm{u}} ^{-}$}, {$\nabla \bm{u}^{-}$}, {$\bm{E} ^{-}$}
and {${\bm{B}} ^{-}$} belong to {$L^{\infty}$}, and when
$ k < k_{0}$, there exists a constant {$\alpha>0$}, such that
$$
\inf_{\bm{0} \neq \bm{\xi}\in \bm{X}^{\bm{u}, 0}} \sup_{\bm{0} \neq \bm{\eta}\in \bm{X}^{\bm{u}, 0}}\frac{\bm{a}_{N}(\bm{\xi}, \bm{\eta}) }{\|\bm{\xi}\|_{\bm{X}} \|\bm{\eta}\|_{\bm{X}}}\geq \alpha>0,
$$
$$
\inf_{\bm{0} \neq \bm{\eta}\in \bm{X}^{\bm{u}, 0}} \sup_{ \bm{0} \neq \bm{\xi} \in \bm{X}^{\bm{u}, 0}}\frac{\bm{a}_{N}(\bm{\xi}, \bm{\eta}) }{\|\bm{\xi}\|_{\bm{X}} \|\bm{\eta}\|_{\bm{X}}}\geq \alpha>0.
$$
where {$\alpha$} depends on the domain {$\Omega$}, the known functions {${\bm{u}} ^{-}$}, {$\bm{E} ^{-}$}, {${\bm{B}} ^{-}$}, but not on the size of time step {$k$}. Furthermore, {$k_{0}$} is a constant which depends on {$\Vert \bm{B}^{-} \Vert_{0, \infty}$}, {$\Vert \bm{u}^{-} \Vert_{0, \infty}$}, {$\Vert \nabla \bm{u}^{-} \Vert_{0, \infty}$}, {$\Vert \bm{E}^{-} \Vert_{0, \infty}$} and {$\Vert \bm{u}^{-} \times \bm{B}^{-} \Vert_{0, \infty}$}.
\end{lemma}

\begin{proof}
Suppose {$k$} is chosen as in the theorem, and take {$\bm{v}=\bm{u}$}, {$\bm{F}=\bm{E}$}, {$\bm{C}=\frac{1}{2}(\bm{B}+k\nabla\times \bm{E})$}. 
%
By assumption, we have
\begin{dgroup*}
\begin{dmath*}
S (\bm{E} ^{-}\times \bm{B}, \bm{v})
+ S (({\bm{u}} ^{-}\times {\bm{B}} ^{-})\times \bm{B}, \bm{v})
+ S (({\bm{u}} ^{-}\times \bm{B})\times {\bm{B}} ^{-}, \bm{v})
+ S (\bm{u}\times \bm{B}^{-}, \bm{u}\times \bm{B}^{-} )
\le C (\|\bm{B}\|^{2}+\|\bm{u}\|^{2}),
\end{dmath*}
\begin{dmath*}
S ({\bm{u}} ^{-}\times \bm{B}, \bm{F})
\le 8 S \|{\bm{u}} ^{-}\|_{0,\infty}^{2} \|\bm{B}\|^{2}
+ {S \over 8}\|\bm{E}\|^{2},
\end{dmath*}
\end{dgroup*}
and 
\begin{align*}
& (\bm{u}\cdot \nabla\bm{u}^{-}, \bm{u})
\leq 3 \|\nabla\bm{u}^{-}\|_{0,\infty}\|\bm{u}\|^{2}, \\
&(\bm{u}^{-}\cdot \nabla\bm{u}, \bm{u})
\leq {1\over 2Re}\|\nabla \bm{u}\|^{2}
+  {8 Re} \|\bm{u}^{-}\|_{0,\infty}^{2} \|\bm{u}\|^{2},  \\
& S (\bm{E}, \bm{v}\times \bm{B}^{-})
\leq {S \over 8}\|\bm{E}\|^{2}
+ 8 S \|\bm{B}^{-}\|_{0,\infty}^{2}\|\bm{v}\|^{2}, \\
& S (\bm{u}\times \bm{B}^{-}, \bm{F})
\leq {S \over 8}\|\bm{E}\|^{2}
+ 8 S \|\bm{B}^{-}\|_{0,\infty}^{2}\|\bm{v}\|^{2}.
\end{align*}
These imply that
\begin{align*}
\bm{a}_{N}(\bm{\xi}; \bm{\eta})
& \geq \frac{1}{2k} (\| \bm{u}\|^{2}+\|\bm{B}\|^{2}) 
+ {1 \over 2Re} |\bm{u}|^{2}_{1}
+ \frac{5}{8} S \|\bm{E}\|^{2}
+ \frac{S k}{Rm} \|\nabla\times \bm{E}\|^{2} 
+ \frac{S}{2 Rm} \|\nabla\cdot \bm{B}\|^{2}  \\
&\geq \frac{1}{2} \min \left \{1, \frac{1}{Re} \right\} \|\bm{u}\|_{ 1,k}^{2}
+ \frac{1}{2} \min \left \{ 1, \frac{S}{Rm} \right\} \|\bm{B}\|_{\mathrm{div},k}^{2}
+ \min \left \{ \frac{5}{8} S, \frac{S}{Rm} \right \} \|\bm{E}\|_{ \mathrm{curl}, k}^{2} \\
& \geq \alpha \Vert \bm{\xi} \Vert^{2}_{\bm{X}}.
\end{align*}
And by definition of {$\bm{v}$}, {$\bm{F}$}, {$\bm{C}$}:
$$
\|\bm{\eta}\|_{\bm{X}} \leq C \|\bm{\xi}\|_{\bm{X}}.
$$
where {$\alpha$} and {$C$} do not depend on 
{$k$}. The other inequality can be proved in a similar
way. 
\end{proof}

\section{Concluding remarks}\label{sec:concluding}
In the discretization of MHD systems, the importance of preserving the
divergence-free condition of magnetic field on the discrete level is
well-established in the literature.  By keeping the electric-field
$\bm E$ as a discretization variable and using mixed finite element
methods together with techniques from discrete differential forms or finite
element exterior calculus \cite{Hiptmair.R.2002a,Arnold.D;Falk.R;Winther.R.2006a,Arnold.D;Falk.R;Winther.R.2010a},
we designed several new finite element discretization schemes that
naturally preserve the divergence-free condition exactly.  We
have rigorously proved that these schemes are well-posed and they also
satisfy desirable energy estimates.  Thanks to the
structure-preserving property of the new schemes, one important
by-product of our analysis is that a class of robust preconditioners can be 
obtained for the linearized systems resulting from this mixed finite element
discretization. For example, the operator form of symmetric Picard
linearization (as described in Problem \ref{prob:picard_explicit}) is
\begin{gather*}\label{picard_evolution}
\left (
\begin{array}{cccc}
{k^{-1}} \mathcal{I}_{\bm{u}} + {\mathcal{D}}  & \mathcal{F} ^{\ast}& 0 &-\mathrm{div}^{\ast}\\
\mathcal{F}  & S \mathcal{I}_{\bm{E}} & - \frac{S}{Rm} \mathrm{curl}^{\ast}& 0\\
0 & - \frac{S}{Rm} \mathrm{curl} & -\frac{S k^{-1} }{Rm} \mathcal{I}_{\bm{B}}& 0 \\
-\mathrm{div} & 0 & 0 & 0
\end{array}
\right )
\left(
\begin{array}{c}
\bm{u} \\ \bm{E} \\ \bm{B}\\ p 
\end{array}
\right)
=
\left(
\begin{array}{c}
{\bm{f}} \\ \bm{r}\\\ - \bm{l} \\ -g
\end{array}
\right).
\end{gather*}
 
We note that the resulting linear system from this discretization is
actually symmetric.  While this symmetry is not critically important
from a practical point of view, it is remarkable that such a property can
be derived for such a complicated nonlinear system.  Using the
well-posed result, Theorem \ref{thm:picard1}, we can naturally design 
a symmetric positive definite
preconditioner and further prove that the resulting preconditioned
iterative method (for example, MINRes) converges uniformly with respect to mesh
parameters. While the other Picard and Newton linearization schemes do
not lead to exactly symmetric coefficient matrices, we can still use the
well-posedness results (Theorems \ref{thm:picard1}, \ref{thm:mixed0}
and \ref{thm:newton1}) to design robust preconditioners for these
systems.  The details on these and other relevant preconditioning
techniques will be reported in another paper
\cite{Hu.K;Hu.X;Ma.Y;Xu.J.2014a}.

We also would like to comment that even though we only consider 
a very special case of MHD equation in this paper, it is natural to generalize
our techniques to other MHD models: the treatment of the magnetic field
and preservation of its divergence-free condition can be done in
exactly the same way while the discretization of the fluid part of the
equations should be handled appropriately.  We will report relevant
results in this direction in future papers.

\section*{Acknowledgements}
The authors would like to thank Long Chen, Xiaozhe Hu, Maximilian Metti, Shuo Zhang, Ludmil Zikatanov for useful discussions and suggestions.

\bibliographystyle{acm}
{\small\bibliography{MHD}{} }  

\begin{thebibliography}{10}

\bibitem{Armero.F;Simo.J.1996a}
{\sc Armero, F., and Simo, J.}
\newblock {Long-term dissipativity of time-stepping algorithms for an abstract
  evolution equation with applications to the incompressible MHD and
  Navier-Stokes equations}.
\newblock {\em Computer Methods in Applied Mechanics and Engineering 131\/}
  (April 1996), 41--90.

\bibitem{Arnold.D;Falk.R;Winther.R.2006a}
{\sc Arnold, D., Falk, R., and Winther, R.}
\newblock {Finite element exterior calculus, homological techniques, and
  applications}.
\newblock {\em Acta numerica 15\/} (May 2006), 1.

\bibitem{Arnold.D;Falk.R;Winther.R.2010a}
{\sc Arnold, D., Falk, R., and Winther, R.}
\newblock {Finite element exterior calculus: from Hodge theory to numerical
  stability}.
\newblock {\em Bulletin of the American Mathematical Society 47}, 2 (Jan.
  2010), 281--354.

\bibitem{Aydn.S.2010a}
{\sc Aydın, S.}
\newblock {Two-level finite element method with a stabilizing subgrid for the
  incompressible MHD equations}.
\newblock {\em International Journal for Numerical Methods in Fluids 62}, 2
  (2010), 188--210.

\bibitem{Badia.S;Codina.R;Planas.R.2013a}
{\sc Badia, S., Codina, R., and Planas, R.}
\newblock {On an unconditionally convergent stabilized finite element
  approximation of resistive magnetohydrodynamics}.
\newblock {\em Journal of Computational Physics 234\/} (Feb. 2013), 399--416.

\bibitem{Balsara.D.2004a}
{\sc Balsara, D.}
\newblock Second-order-accurate schemes for magnetohydrodynamics with
  divergence-free reconstruction.
\newblock {\em The Astrophysical Journal Supplement Series 151\/} (2004),
  149--184.

\bibitem{Balsara.D;Kim.J.2004a}
{\sc Balsara, D., and Kim, J.}
\newblock {A comparison between divergence-cleaning and staggered-mesh
  formulations for numerical magnetohydrodynamics}.
\newblock {\em The Astrophysical Journal 602\/} (2004), 1079--1090.

\bibitem{Balsara.D;Spicer.D.1999a}
{\sc Balsara, D., and Spicer, D.}
\newblock {A staggered mesh algorithm using high order Godunov fluxes to ensure
  solenoidal magnetic fields in magnetohydrodynamic simulations}.
\newblock {\em Journal of Computational Physics 149\/} (1999), 270--292.

\bibitem{Bandaru.V;Boeck.T;Krasnov.D;Schumacher.J.1999a}
{\sc Bandaru, V., Boeck, T., Krasnov, D., and Schumacher, J.}
\newblock {Numerical computation of liquid metal MHD duct flows at finite
  magnetic Reynolds number}.
\newblock {\em pamir.sal.lv\/} (1999).

\bibitem{Banas.L;Prohl.A.2010a}
{\sc Ba\u{n}as, L., and Prohl, A.}
\newblock {Convergent finite element discretization of the multi-fluid
  nonstationary incompressible magnetohydrodynamics equations}.
\newblock {\em Mathematics of Computation 79}, 272 (2010), 1957--1999.

\bibitem{Boffi.D;Brezzi.F;Fortin.M.2013a}
{\sc Boffi, D., Brezzi, F., and Fortin, M.}
\newblock {\em {Mixed Finite Element Methods and Applications}}.
\newblock Springer, 2013.

\bibitem{Bossavit.A.1998a}
{\sc Bossavit, A.}
\newblock {\em Computational Electromagnetism}.
\newblock Academic Press (Boston), 1998.

\bibitem{Bossavit.A.2005a}
{\sc Bossavit, A.}
\newblock {Discretization of Electromagnetic Problems: The `` Generalized
  Finite Differences '' Approach}.
\newblock {\em Handbook of numerical analysis XIII}, 04 (2005).

\bibitem{Brackbill.J.1985a}
{\sc Brackbill, J.}
\newblock Fluid modeling of magnetized plasmas.
\newblock {\em Space Plasma Simulations 42\/} (1985), 153--167.

\bibitem{Brackbill.J;Barnes.D.1980a}
{\sc Brackbill, J.~U., and Barnes, D.~C.}
\newblock {The effect of nonzero $\nabla\cdot B$ on the numerical solution of
  the magnetohydrodynamic equations}.
\newblock {\em Journal of Computational Physics 430\/} (1980), 426--430.

\bibitem{Brezzi.F.1974a}
{\sc Brezzi, F.}
\newblock {On the existence, uniqueness and approximation of saddle-point
  problems arising from Lagrangian multipliers}.
\newblock {\em ESAIM: Mathematical Modelling and Numerical
  Analysis-Mod\'{e}lisation Math\'{e}matique et Analyse Num\'{e}rique 8\/}
  (1974), 129--151.

\bibitem{Bris.C;Lelievre.T.2006a}
{\sc Bris, C., and Leli\`{e}vre, T.}
\newblock {\em {Mathematical Methods for the Magnetohydrodynamics of Liquid
  Metals}}.
\newblock Oxford University Press, USA, 2006.

\bibitem{Cai.W;Wu.J;Xin.J.2013a}
{\sc Cai, W., Wu, J., and Xin, J.}
\newblock {Divergence-free H(div)-conforming hierarchical bases for
  magnetohydrodynamics (MHD)}.
\newblock {\em Communications in Mathematics and Statistics 1\/} (2013),
  19--35.

\bibitem{Clarke.D;Norman.M;Burns.J.1986a}
{\sc Clarke, D., Norman, M., and Burns, J.}
\newblock {Numerical simulations of a magnetically confined jet}.
\newblock {\em The Astrophysical Journal 311\/} (1986), 63--67.

\bibitem{Cockburn.B;Li.F;Shu.C.2004a}
{\sc Cockburn, B., Li, F., and Shu, C.}
\newblock {Locally divergence-free discontinuous Galerkin methods for the
  Maxwell equations}.
\newblock {\em Journal of Computational Physics 194}, 2 (Mar. 2004), 588--610.

\bibitem{Codina.R;Hernandez.N.2011a}
{\sc Codina, R., and Hern\'{a}ndez, N.}
\newblock {Approximation of the thermally coupled MHD problem using a
  stabilized finite element method}.
\newblock {\em Journal of Computational Physics 230}, 4 (Feb. 2011),
  1281--1303.

\bibitem{Conraths.H.1996a}
{\sc Conraths, H.}
\newblock {Eddy current and temperature simulation in thin moving metal
  strips}.
\newblock {\em International Journal for Numerical Methods in Engineering 39\/}
  (January 1996), 141--163.

\bibitem{Dai.W;Woodward.P.1998a}
{\sc Dai, W., and Woodward, P.}
\newblock {A simple finite difference scheme for multidimensional
  magnetohydrodynamical Equations}.
\newblock {\em Journal of Computational Physics 142\/} (May 1998), 331--369.

\bibitem{Dai.W;Woodward.P.1998b}
{\sc Dai, W., and Woodward, P.}
\newblock {On the divergence-free condition and conservation laws in numerical
  simulations for supersonic magnetohydrodynamic flows}.
\newblock {\em The Astrophysical Journal 494}, 1 (1998).

\bibitem{Davidson.P.2001a}
{\sc Davidson, P.}
\newblock {\em {An Introduction to Magnetohydrodynamics}}.
\newblock Cambridge University Press, 2001.

\bibitem{Dios.B;Brezzi.F;Marini.L;Xu.J;Zikatanov.L.2014a}
{\sc de~Dios, B., Brezzi, F., Marini, L., Xu, J., and Zikatanov, L.}
\newblock {A simple preconditioner for a discontinuous Galerkin method for the
  Stokes problem}.
\newblock {\em Journal of Scientific Computing 58}, 3 (Aug. 2014), 517--547.

\bibitem{Dedner.A;Kemm.F;Kroner.D;Munz.C;Schnitzer.T;Wessenberg.M.2002a}
{\sc Dedner, A., Kemm, F., Kr\"{o}ner, D., Munz, C., Schnitzer, T., and
  Wessenberg, M.}
\newblock {Hyperbolic divergence cleaning for the MHD equations}.
\newblock {\em Journal of Computational Physics 175}, 2 (Jan. 2002), 645--673.

\bibitem{Demkowicz.L;Vardapetyan.L.1998a}
{\sc Demkowicz, L., and Vardapetyan, L.}
\newblock Modeling of electromagnetic absorption/scattering problems using
  hp-adaptive finite elements.
\newblock {\em Computer Methods in Applied Mechanics and Engineering 152\/}
  (1998), 103--124.

\bibitem{C.R.DeVore.1991a}
{\sc DeVore, C.}
\newblock Flux-corrected transport techniques for multidimensional compressible
  magnetohydrodynamics.
\newblock {\em Journal of Computational Physics 92\/} (1991), 142.

\bibitem{Evans.C;Hawley.J.1988a}
{\sc Evans, C., and Hawley, J.}
\newblock {Simulation of magnetohydrodynamic flows-A constrained transport
  method}.
\newblock {\em The Astrophysical Journal 332\/} (1988), 659--677.

\bibitem{Fautrelle.Y.1981a}
{\sc Fautrelle, Y.}
\newblock {Analytical and numerical aspects of the electromagnetic stirring
  induced by alternating magnetic fields}.
\newblock {\em Journal of Fluid Mechanics 102\/} (1981), 405--430.

\bibitem{Fey.M;Torrilhon.M.2003a}
{\sc Fey, M., and Torrilhon, M.}
\newblock A constrained transport upwind scheme for divergence-free advection.
\newblock {\em Hyperbolic Problems: Theory, Numerics, Applications\/} (2003),
  529--538.

\bibitem{Gerbeau.J;Lelievre.T;Bris.C.2003a}
{\sc Gerbeau, J., Leli\`{e}vre, T., and Bris, C.}
\newblock {Simulations of MHD flows with moving interfaces}.
\newblock {\em Journal of Computational Physics 184\/} (January 2003),
  163--191.

\bibitem{Girault.V;Raviart.P.1986a}
{\sc Girault, V., and Raviart, P.}
\newblock {\em {Finite Element Methods for Navier-Stokes Equations: Theory and
  Algorithms}}.
\newblock Springer, 1986.

\bibitem{Guermond.J;Minev.P.2003a}
{\sc Guermond, J.~L., and Minev, P.~D.}
\newblock {Mixed finite element approximation of an MHD problem involving
  conducting and insulating regions: The 3D case}.
\newblock {\em Numerical Methods for Partial Differential Equations 19}, 6
  (Nov. 2003), 709--731.

\bibitem{Gunzburger.M;Meir.A;Peterson.J.1991a}
{\sc Gunzburger, M., Meir, A., and Peterson, J.}
\newblock {On the existence, uniqueness, and finite element approximation of
  solutions of the equations of stationary, incompressible
  magnetohydrodynamics}.
\newblock {\em Mathematics of Computation 56}, 194 (1991), 523--563.

\bibitem{Hasler.U;Schneebeli.A;Schotzau.D.2004a}
{\sc Hasler, U., Schneebeli, A., and Sch\"{o}tzau, D.}
\newblock Mixed finite element approximation of incompressible mhd problems
  based on weighted regularization.
\newblock {\em Applied Numerical Mathematics 51\/} (2004), 19--45.

\bibitem{Helzel.C;Rossmanith.J;Taetz.B.2011a}
{\sc Helzel, C., Rossmanith, J.~A., and Taetz, B.}
\newblock An unstaggered constrained transport method for the 3{D} ideal
  magnetohydrodynamic equations.
\newblock {\em Journal of Computational Physics 230\/} (2011), 3803--3829.

\bibitem{Hiptmair.R.2002a}
{\sc Hiptmair, R.}
\newblock {Finite elements in computational electromagnetism}.
\newblock {\em Acta Numerica 11}, July 2003 (July 2002), 237--339.

\bibitem{Hiptmair.R;Heumann.H;Mishra.S;Pagliantini.C.2014a}
{\sc Hiptmair, R., Heumann, H., Mishra, S., and Pagliantini, C.}
\newblock Discretizing the advection of differential forms.
\newblock In {\em ICERM Topical Workshop\/} (May 2014), Robust discretization
  and fast solvers for computable multi-physics models, ICERM.

\bibitem{Houston.P;Schoetzau.D;Wei.X.2008a}
{\sc Houston, P., Sch\"{o}etzau, D., and Wei, X.}
\newblock {A mixed DG method for linearized incompressible
  magnetohydrodynamics}.
\newblock {\em Journal of Scientific Computing 40\/} (July 2009), 281--314.

\bibitem{Hu.K;Hu.X;Ma.Y;Xu.J.2014a}
{\sc Hu, K., Hu, X., Ma, Y., and Xu, J.}
\newblock Robust preconditioners for the magnetohydrodynmaics system.
\newblock Manuscript in preparation, 2014.

\bibitem{Ida.N;Bastos.J.1997a}
{\sc Ida, N., and Bastos, J.~a.~P.}
\newblock {\em {Electromagnetics and Calculation of Fields}}, 2nd~ed.
\newblock Springer, 1997.

\bibitem{Jackson.J.1975a}
{\sc Jackson, J.}
\newblock {\em Classical Electrodynamics}, 2nd~ed.
\newblock John Wiley \& Sons, 1975.

\bibitem{Jardin.S.2010a}
{\sc Jardin, S.}
\newblock {\em {Computational Methods in Plasma Physics}}.
\newblock CRC Press, 2010.

\bibitem{Jiang.B;Wu.J;Povinelli.L.1996a}
{\sc Jiang, B., Wu, J., and Povinelli, L.}
\newblock {The origin of spurious solutions in computational electromagnetics}.
\newblock {\em Journal of Computational Physics 125\/} (1996), 104--123.

\bibitem{Lee.Y;Xu.J;Zhang.C.2011a}
{\sc Lee, Y., Xu, J., and Zhang, C.}
\newblock {Global existence, uniqueness and optimal solvers of discretized
  viscoelastic flow models}.
\newblock {\em Mathematical Models and Methods in Applied Sciences 21\/} (Aug.
  2011), 1713--1732.

\bibitem{Li.F;Shu.C.2005a}
{\sc Li, F., and Shu, C.}
\newblock {Locally divergence-free discontinuous Galerkin methods for MHD
  equations}.
\newblock {\em Journal of Scientific Computing\/} (2005).

\bibitem{Li.F;Xu.L.2012a}
{\sc Li, F., and Xu, L.}
\newblock {Arbitrary order exactly divergence-free central discontinuous
  Galerkin methods for ideal MHD equations}.
\newblock {\em Journal of Computational Physics 231\/} (2012), 2655--2675.

\bibitem{Liu.C.2009a}
{\sc Liu, C.}
\newblock {Energetic variational approaches in complex fluids}.
\newblock In {\em Multi-Scale Phenomena in Complex Fluids: Modeling, Analysis
  and Numerical Simulations}. World Scientific Publishing Company, 2009.

\bibitem{Liu.J;Wang.W.2001a}
{\sc Liu, J., and Wang, W.}
\newblock {An energy-preserving MAC-Yee scheme for the incompressible MHD
  equation}.
\newblock {\em Journal of Computational Physics 174}, 1 (Nov. 2001), 12--37.

\bibitem{Liu.J;Wang.W.2004a}
{\sc Liu, J., and Wang, W.}
\newblock Energy and helicity preserving schemes for hydro{-} and
  magnetohydro{-}dynamics flows with symmetry.
\newblock {\em Journal of Computational Physics 200\/} (May 2004), 8--33.

\bibitem{Londrillo.P;Zanna.L.2004a}
{\sc Londrillo, P., and Zanna, L.}
\newblock {On the divergence-free condition in {G}odunov-type schemes for ideal
  magnetohydrodynamics: the upwind constrained transport method}.
\newblock {\em Journal of Computational Physics 195\/} (March 2004), 17--48.

\bibitem{Monk.P.2003a}
{\sc Monk, P.}
\newblock {\em {Finite Element Methods for Maxwell's equations}}.
\newblock Oxford University Press, 2003.

\bibitem{Nedelec.J.1980a}
{\sc N\'{e}d\'{e}lec, J.}
\newblock Mixed finite elements in $\mathbb{R}^{3}$.
\newblock {\em Numerische Mathematik 35\/} (1980), 315--341.

\bibitem{Nedelec.J.1986a}
{\sc N\'{e}d\'{e}lec, J.}
\newblock A new family of mixed finite elements in $\mathbb{R}^{3}$.
\newblock {\em Numerische Mathematik 50\/} (1986), 57--81.

\bibitem{Pekmen.B;Tezer-Sezgin.M.2013a}
{\sc Pekmen, B., and Tezer-Sezgin, M.}
\newblock {DRBEM solution of incompressible MHD flow with magnetic potential}.
\newblock {\em Computer Modeling in Engineering \& Sciences 96}, 4 (2013),
  275--292.

\bibitem{Powell.K.1997a}
{\sc Powell, K.}
\newblock {An approximate Riemann solver for magnetohydrodynamics}.
\newblock {\em Upwind and High-Resolution Schemes\/} (1997), 570--583.

\bibitem{Prohl.A.2008a}
{\sc Prohl, A.}
\newblock {Convergent finite element discretizations of the nonstationary
  incompressible magnetohydrodynamics system}.
\newblock {\em Mathematical Modelling and Numerical Analysis 42\/} (2008),
  1065--1087.

\bibitem{Raviart.P;Thomas.J.1977a}
{\sc Raviart, P., and Thomas, J.}
\newblock A mixed finite element method for second order elliptic problems.
\newblock {\em Lecture Notes in Mathematics 606\/} (1977), 292--315.

\bibitem{Rossmanith.J.2006a}
{\sc Rossmanith, J.~A.}
\newblock {An unstaggered, high-resolution constrained transport method for
  magnetohydrodynamic flows}.
\newblock {\em SIAM Journal on Scientific Computing 28}, 5 (September 2006),
  1766--1797.

\bibitem{Ryu.D;Miniati.F;Jones.T;Frank.A.1998a}
{\sc Ryu, D., Miniati, F., Jones, T., and Frank, A.}
\newblock {A divergence-free upwind code for multidimensional
  magnetohydrodynamic flows}.
\newblock {\em The Astrophysical Journal 509\/} (December 1998), 244--255.

\bibitem{Salah.N;Soulaimani.A;Habashi.W.2001a}
{\sc Salah, N., Soulaimani, A., and Habashi, W.}
\newblock {A finite element method for magnetohydrodynamics}.
\newblock {\em Computer Methods in Applied Mechanics and Engineering 190\/}
  (2001), 5867--5892.

\bibitem{Salah.N;Soulaimani.A.1999a}
{\sc Salah, N., Soulaimani, A., Habashi, W., and Fortin, M.}
\newblock {A conservative stabilized finite element method for the
  magneto-hydrodynamic equations}.
\newblock {\em International Journal for Numerical methods in fluids 554},
  October 1997 (1999), 535--554.

\bibitem{Schneebeli.A;Schotzau.D.2003a}
{\sc Schneebeli, A., and Sch\"{o}tzau, D.}
\newblock {Mixed finite elements for incompressible magneto-hydrodynamics}.
\newblock {\em Comptes Rendus Mathematique 337}, 1 (2003), 71--74.

\bibitem{Schotzau.D.2004a}
{\sc Sch\"{o}tzau, D.}
\newblock {Mixed finite element methods for stationary incompressible
  magneto--hydrodynamics}.
\newblock {\em Numerische Mathematik\/} (2004), 771--800.

\bibitem{Shadid.J;Pawlowski.R;Banks.J;Chacon.L;Lin.P;Tuminaro.R.2010a}
{\sc Shadid, J., Pawlowski, R., Banks, J., Chacon, L., Lin, P., and Tuminaro,
  R.}
\newblock Towards a scalable fully-implicit fully-coupled resistive {MHD}
  formulation with stabilized {FE} methods.
\newblock {\em Journal of Computational Physics 229\/} (2010), 7649--7671.

\bibitem{Toth.G.2000a}
{\sc T\'{o}th, G.}
\newblock {The $\nabla\cdot B= 0$ constraint in shock-capturing
  magnetohydrodynamics codes}.
\newblock {\em Journal of Computational Physics\/} (2000), 1--24.

\bibitem{Vardapetyan.L;Demkowicz.L.1999a}
{\sc Vardapetyan, L., and Demkowicz, L.}
\newblock hp-adaptive finite elements in electromagnetics.
\newblock {\em Computer Methods in Applied Mechanics and Engineering 169\/}
  (February 1999), 331--344.

\bibitem{Wiedmer.M.2000a}
{\sc Wiedmer, M.}
\newblock {Finite element approximation for equations of magnetohydrodynamics}.
\newblock {\em Mathematics of Computation 69\/} (2000), 83--101.

\bibitem{Yakovlev.S;Xu.L;Li.F.2013a}
{\sc Yakovlev, S., Xu, L., and Li, F.}
\newblock {Locally divergence-free central discontinuous Galerkin methods for
  ideal MHD equations}.
\newblock {\em Journal of Computational Science 4\/} (2013), 80--91.

\bibitem{Ye.X;Hall.C.1997a}
{\sc Ye, X., and Hall, C.~A.}
\newblock {A discrete divergence-free basis for finite element methods}.
\newblock {\em Numerical Algorithms 16\/} (1997), 365--380.

\bibitem{K.S.Yee.1966a}
{\sc Yee, K.}
\newblock Numerical solution of initial boundary value problems involving
  maxwell's equations in isotropic media.
\newblock {\em Journal of Computational Physics 14\/} (1966), 302.

\bibitem{Zhang.S.2012a}
{\sc Zhang, S.}
\newblock {Bases for C0-P1 divergence-free elements and for C1-P2 finite
  elements on union jack grids}.
\newblock Submitted, 2012.

\end{thebibliography}

\end{document}